\documentclass[10pt,reqno]{amsart}
\usepackage{bbm}
\usepackage{amsmath,amsthm,amsfonts,amssymb,amscd,mathrsfs}
\usepackage{amssymb,amsfonts,amsthm, color, algorithm} %\usepackage{psfrag}%algpseudocode
\usepackage{bm}
\usepackage{graphicx}
 \usepackage{graphicx}  

\usepackage{epsfig}
\usepackage{paralist}
\usepackage{color}

\usepackage{hyperref}
\hypersetup{
    colorlinks=true,
    linkcolor=blue,
    filecolor=magenta,
    urlcolor=cyan,
    pdftitle={Sharelatex Example},
    pdfpagemode=FullScreen,
}

\theoremstyle{plain}
\newtheorem{theorem}{Theorem}[section]

\newtheorem{lemma}[theorem]{Lemma}
\newtheorem{cor}[theorem]{Corollary}
\newtheorem{prop}[theorem]{Proposition}

\newtheorem{thmx}{Theorem}[section]

\theoremstyle{definition}

\newtheorem{remark}[theorem]{Remark}

\numberwithin{equation}{section}

\def\max{{\text{max}}}

\def\N{\mathbb{N}}
\def\P{\mathbb{P}}

\def\R{\mathbb{R}}

\def\Z{\mathbb{Z}}

\def\K{\mathbb{K}}

\def\cF{{\mathcal F}}
\def\cM{{\mathcal M}}

\def\cX{{\mathcal X}}
\def\cY{{\mathcal Y}}

\def\om{\omega}
\def\Om{\Omega}
\def\de{\delta}

\def\var{\varepsilon}
\begin{document}

\title[Intermittency in Markov random networks]{
Intermittent synchronization in finite-state random networks under Markov perturbations}

\author[A. Berger]{Arno Berger}
\address{A. Berger: Department of Mathematical \& Statistical Sciences,
    University of Alberta, Edmonton, Alberta, Canada T6G 2G1} \email{berger@ualberta.ca}

\author[H. Qian]{Hong Qian}
\address{H. Qian: Department of Applied Mathematics, University of Washington, Seattle, WA 98195,
    USA} \email{hqian@u.washington.edu}

\author[S. Wang]{Shirou Wang*}
\address{S. Wang: Department of Mathematical \& Statistical Sciences,
    University of Alberta, Edmonton, Alberta, Canada T6G 2G1} \email{shirou@ualberta.ca}

\author[Y. Yi]{Yingfei Yi}
\address{Y. Yi: Department of Mathematical
    \& Statistical Sciences, University of Alberta, Edmonton, Alberta,
    Canada T6G 2G1,  and School of Mathematics, Jilin University,
    Changchun 130012, PRC} \email{yingfei@ualberta.ca}

\begin{flushleft}
{\it Comm. Math. Phys.},  to appear

\vspace{5mm}
\end{flushleft}



\thanks{* Corresponding author.}

\thanks{A.\ Berger was supported by an NSERC Discovery Grant. H.\ Qian
  was supported by the Olga Jung Wan Endowed Professorship. S.\ Wang
  was partially supported by PIMS PTCS, NSFC grants 11771026,
  11471344, and by the PIMS site at the University of Washington
  through NSF grant DMS-1712701. Y.\ Yi was partially supported by an
  NSERC Discovery Grant, a faculty development grant from the
  University of Alberta, and a Scholarship from Jilin University. All authors were supported in part by a PIMS CRG
 grant.}

\begin{abstract} By introducing extrinsic noise as well as intrinsic
  uncertainty into a network with stochastic events, this paper
  studies the dynamics of the resulting {\em Markov random network}
  and characterizes a novel phenomenon of intermittent synchronization
  and desynchronization that is due to an interplay of the two forms
  of randomness in the system. On a finite state space and in discrete
  time, the network allows for unperturbed (or ``deterministic'')
  randomness that represents the extrinsic noise but also for small
  intrinsic uncertainties modelled by a Markov perturbation. It is shown that if the deterministic
  random network is synchronized (resp., uniformly
  synchronized), then for almost all realizations of its
  extrinsic noise the stochastic trajectories of the perturbed network
  synchronize along almost all (resp., along all) time sequences
  after a certain time, with high probability. That is, both
  the probability of synchronization and the proportion of time spent
  in synchrony are arbitrarily close to one. Under smooth Markov
  perturbations, high-probability synchronization and low-probability desynchronization occur intermittently in time.
  If the perturbation is $C^m$ ($m \ge 1$) in
  $\var$, where $\var$ is a perturbation parameter, then the relative
  frequencies of synchronization with probability $1-O(\var^{\ell})$ and
  of desynchronization with probability $O(\var^{\ell})$ can both be
  precisely described for $1\le \ell\le m$ via an asymptotic
  expansion of the invariant distribution. Existence and
  uniqueness of invariant distributions are established, as well as their convergence
  as $\var \to 0$. An explicit asymptotic expansion is
  derived. Ergodicity of the extrinsic noise dynamics is seen to be crucial for
  the characterization of (de)synchronization sets
  and their respective relative frequencies. An example of a smooth
  Markov perturbation of a synchronized probabilistic Boolean network is provided to
  illustrate the intermittency between high-probability
  synchronization and low-probability desynchronization.
\end{abstract}

\maketitle

\section{Introduction}
Physics, engineering, and sciences of complex systems and
processes often encounter network dynamics that are subject to
uncertainties from ``individuals'' in a population but also to random
influences from the surrounding environment, referred to respectively
as {\em intrinsic\/} and {\em extrinsic\/} noise. While the former is
often due to internal complexities of the individuals one studies, the
latter reflects the unpredictable world one lives in
\cite{MQY,ye-qian,ye-wang-qian}. This paper proposes and develops the
theory of {\em Markov random networks\/} as an appropriate conceptual
framework that distinguishes intrinsic and extrinsic noise and
incorporates them into a comprehensive dynamical theory. It has been speculated
\cite{J} that whereas extrinsic noise may cause ``noise-induced
synchronization'', a familiar scenario in the context of random
dynamical systems \cite{FGS2017,HQWYY,Newman1}, intrinsic noise
will drive synchronized individuals apart. The present work provides the
first systematic analysis of this phenomenon. Adopting a simple
discrete-time, finite-state framework for the stochastic dynamics on a
network allows for a treatment that goes well beyond the standard approach
available for non-autonomous or random dynamics on a more general
space, and in particular enables an in-depth analysis of the
intermittency between synchronization and desynchronization.

Let  $S=\{s_1,\ldots,s_k\}$ be a finite set endowed with the discrete topology, and $\Theta:=(\Om,\cF,\mu,\theta)$ an invertible metric
dynamical system, that is, $(\Om,\cF,\mu)$ is a standard measure space with $\mu
(\Om) = 1$, and $\theta:\Om\to\Om$ is an invertible ergodic
$\mu$-preserving transformation. The view adopted herein
is that $\Theta$ provides a model of the extrinsic noise. Call a stochastic process $\cX =
(X_n)_{n\in \N_0}$ with state space $S\times \Om$ a {\em
  Markov random network\/} (MRN) if
\begin{enumerate}
\item[(MRN1)] $\cX$ is measurable in distribution, that is,  $\om \mapsto {\P}
  \bigl\{X_{n}=(s_i,\theta^{n}\om)|X_0=(s_j,\om)\bigr\}$ is
  $\cF$-measurable for all $n\in \N_0$ and $i,j \in \{1, \ldots ,
  k\}$;
\item[(MRN2)] $\cX$ is stochastic over $\Theta$, that is,
%\begin{equation}\label{skew_product-property}
$$
\sum\nolimits_{i=1}^k{\P}\big\{X_{n}=(s_i,\theta^{n-m}\om)|X_m=(s_j,\om)\big\}=1
\quad\forall j\in\{1,\ldots,k\}
$$
%%\end{equation}
for all $n\ge m$ and $\mu$-a.e.\ $\om \in \Om$;
\item[(MRN3)] $\cX$ has the Markov property over $\Theta$, that is,
\begin{align*}%\label{markov_property}
& {\P}
\big\{X_{n+1}=(s_{i_{n+1}},\theta^{n+1}\omega)\, |\,
X_n=(s_{i_n},\theta^n\omega)\big\}\\
&=   {\P} \big\{X_{n+1}=(s_{i_{n+1}},\theta^{n+1}\omega)\,
|\,X_0=(s_{i_0},\omega),\ldots,
X_n=(s_{i_n},\theta^n\omega)\big\} \nonumber
\end{align*}
for all $n\in \N_0$, $i_0, \ldots, i_n, i_{n+1}\in \{1, \ldots , k\}$
and $\mu$-a.e.\ $\om \in \Om$.
\end{enumerate}
Given an MRN $\cX$, let
$$
p_{i,j} (n, \om) = {\P}
\bigl\{X_{n}=(s_i,\theta^{n}\om)|X_0=(s_j,\om)\bigr \}
\quad \forall i,j \in \{1, \ldots, k\}, n \in \N_0, \om \in \Om \, .
$$
By properties (MRN 2) and (MRN 3), for every $n\in \N_0$ and $\mu$-a.e.\
$\om \in \Om$ the matrix
$$
P_{\cX} (n, \omega) :=\bigl( p_{i,j} (n,\om)\bigr)_{1\le
  i,j \le k} \in [0,1]^{k\times k}
$$
is (column-)stochastic, and the $2^{\N_0}\otimes \cF$-measurable
function $P_{\cX}:\N_0 \times \Om\to \R^{k\times k}$ has the {\em
  cocycle property}, that is, for $\mu$-a.e.\ $\om \in \Om$,
$$
P_{\cX}(m+n, \om) = P_{\cX}(m,\theta^n \om) P_{\cX}(n,\om) \quad \forall m,n \in \N_0 \, .
$$
Call $P_{\cX}$, an example of a Markov cocycle, the {\em transition
cocycle\/} of $\cX$; see also Proposition \ref{prop210} below. Furthermore, let
$$
p_i(n,\om) =  {\P}   \bigl\{X_{n}=(s_i,\theta^{n}\om)| X_n \in S \times
\{\theta^n \om \} \bigr\}
\quad \forall i\in \{1, \ldots, k\}, n \in \N_0, \om \in \Om \, .
$$
With this, for every $n\in \N_0$ and $\mu$-a.e.\ $\om \in \Om$ the vector
$$
p_{\cX} (n,\om) :=  \bigl( p_i (n,\om)\bigr)_{1\le i \le k} \in [0,1]^k
$$
can be thought of as the distribution of $\cX$ on the fibre $S\times
\{\theta^n \omega\}$; see also Proposition \ref{prop210} below.

A special class of MRN are {\em deterministic random networks\/}
(DRN) for which, by definition, the transition cocycle is {\em deterministic\/} in the
sense that for every $n\in \N_0$ and $\mu$-a.e.\ $\om \in \Om$ each
$p_{i,j} (n,\om)$ equals either $0$ or $1$, that is, $P_{\cX} (n,\om) \in
\{0,1\}^{k\times k}$. Throughout this paper, for the sake of clarity,
a DRN typically is denoted by $\cX^0$, and its transition cocycle by
$P^0 := P_{\cX^0}$. Usage of the term ``deterministic'' emphasizes the
absence of (internal) stochasticity between individual states in
$S$. To put this terminology into context, note that every DRN $\cX^0$
uniquely defines a so-called {\em discrete-time finite-state random dynamical system\/} (dtfs-RDS) on
$S$ over $\Theta$. Specifically, for all $n\in \N_0$ and $\mu$-a.e. $\om \in
\Om$ define maps $T_{\cX^0} = T_{\cX^0}(n,\om): S\to S$ such that for
all $i,j\in \{1,\ldots, k\}$,
\begin{eqnarray}\label{deter_cocy-RDS}
    T_{\cX^0} (n,\om)(s_j)=s_i \quad {\text {if and only if}} \quad
    p^0_{i,j}(n,\om)=1\, .
\end{eqnarray}
It is easy to see that $T_{\cX^0}$ inherits the cocycle property from $P^0$, that is,
$$
T_{\cX^0} (m+n , \om) = T_{\cX^0} (m,\theta^n \om) \circ T_{\cX^0}
(n,\om)
$$
for all $m,n\in \N_0$ and $\mu$-a.e.\ $\om \in \Om$, and hence is a
dtfs-RDS in the sense of \cite{arnold98book}. Conversely, every dtfs-RDS on
$S$ over $\Theta$ induces a DRN. Prominent examples of DRN are, for
instance, probabilistic Boolean networks that model gene regulations \cite{SD}.

Adopting terminology from \cite{HQWYY} say that a DRN $\cX^0$ is {\em
  synchronized\/} if there exists an $\cF$-measurable function $N:\Om
\to \N$ such that for all $i,j\in \{1, \ldots , k\}$ and $\mu$-a.e.\
$\om \in \Om$,
$$
T_{\cX^0} (n,\om) (s_i) = T_{\cX^0} (n,\omega) (s_j) \quad \forall n\ge N(\om)\, ;
$$
if $N$ is constant $\mu$-a.e.\ on $\Om$ then $\cX^0$ is {\em uniformly
 synchronized}. Synchronization of DRN is characterized in \cite{HQWYY}
in terms of the Lyapunov exponents of the associated
(deterministic) transition cocycle $P^0$; in particular, it is shown that
$\cX^0$ is synchronized if and only if the Lyapunov exponent $0$ of
$P^0$ is simple.

The present paper focuses on MRN that are
small perturbations of synchronized DRN. The main goal is to analyze
the effect small perturbations have on synchronization. To be specific,
given a DRN $\cX^0$, call a family $\{\cX^{\varepsilon}:
\varepsilon\ge 0\}$ of MRN a {\em Markov perturbation\/} of $\cX^0$ if
for $\mu$-a.e.\ $\om \in \Om$
$$
%\begin{equation}\label{mark_pert}
 |P_{\cX^{\varepsilon}} (1,\omega) -P^0(1,\omega)
 |\le\varepsilon \quad\forall \varepsilon \ge 0 \, ;
$$
% \end{equation}
here $|\, \cdot \,|$ denotes the norm on $\R^{k\times k}$ induced by
the $\ell^1$-norm on $\R^k$.

Markov perturbations of a DRN have a clear physical
meaning: With small intrinsic noise added to the network (which itself
allows only for extrinsic noise), transitions between
states now occur with probabilities at most $O(\var)$ or
at least $1-O(\var)$, rather than being impossible or certain,
respectively. Pertinent examples include probabilistic Boolean networks
with random gene perturbations, where each gene has a small
probability of flipping its value --- naturally, this yields a Markov
perturbation of the original DRN.
In general, existence, uniqueness, and attractiveness of an
invariant distribution of an MRN all require certain monotonicity or
Perron--Frobenius-type assumptions \cite{arnold-ch98, arnold94positive}. As the first main result of this
work illustrates, however,
these assumptions are satisfied automatically for any
Markov perturbation of a synchronized DRN. In the statement,
$\Sigma_1^+$ denotes the set of all probability distributions (or
vectors) on $S$, and $e_1, \ldots , e_k$ is the canonical basis of $\R^k$; see Section \ref{sec2}
below for precise definitions of all technical terms.

\begin{thmx} \label{thm:uniq_inv_distri}
Let $\{\cX^{\var}\}$ be a Markov perturbation of a synchronized DRN
$\cX^0$. Then, for every
sufficiently small $\var \ge 0$, there exists an invariant
distribution $p_{\var}:\Om \to \Sigma_1^+$ of $\cX^{\var}$, i.e.,
$P_{\cX^{\var}} (1,\om) p_{\var} (\om) = p_{\var} (\theta \om)$, with the following properties:
\begin{enumerate}
\item $p_{\var}$ is pull-back attracting for $\cX^{\var}$, that is,
  for every $q\in\Sigma_1^+$ and 
for $\mu$-a.e.\ $\om \in \Om$,
$$ \lim\nolimits_{n \to \infty} |P_{\cX^{\var}}
(n,\theta^{-n}\om) q - p_{\var}(\om)| = 0\, ;$$
\item $p_{\var}$ is forward attracting for $\cX^{\var}$, that is, for every $q\in\Sigma_1^+$ and 
$\mu$-a.e.\ $\om \in \Om$,
$$ 
\lim\nolimits_{n \to \infty} |P_{\cX^{\var}} (n,\om) q  - p_{\var}
(\theta^n \om)| = 0 \, ;
$$
\item $p_{\var}$ is continuous at $\var = 0$, that is, there exists an
  $\cF$-measurable function $J:\Om \to \{1, \ldots , k\}$ such that
  for $\mu$-a.e.\ $\om \in \Om$,
\begin{equation}\label{limit-distribution}
\lim\nolimits_{\var \to 0} p_{\var} (\om) = e_{J(\om)} = p_0 (\om) \, .
\end{equation}
\end{enumerate}
Moreover, if $\cX^0$ is uniformly synchronized then the convergence in
{\rm (\ref{limit-distribution})} is uniform $\mu$-a.e.\ on $\Om$.
\end{thmx}

Theorem \ref{thm:uniq_inv_distri} is proved by analyzing Lyapunov exponents of the
transition cocycle $P_{\cX^{\var}}$. Specifically, it is the continuity of
Lyapunov exponents of $P_{\cX^{\var}}$ as $\var\to 0$ that yields the
simplicity of the Lyapunov exponent $0$ of $P_{\cX^{\var}}$ for
sufficiently small $\var$. The reader may want to recall that
Lyapunov exponents of  random cocycles in general are discontinuous with respect to  generic
perturbations \cite{huang-yi2012, viana2020}. As will be seen, the continuity of Lyapunov exponents
in the context of Theorem \ref{thm:uniq_inv_distri} is due to a
contraction property of the unperturbed cocycle $P^0$.
The essence of Theorem \ref{thm:uniq_inv_distri}, then, is that this contraction property
is preserved under Markov perturbations at the level of
distributions. At the level of stochastic trajectories, the property is reflected
by a high probability of synchronization. The second main result
states that high-probability synchronization among stochastic
trajectories prevails for large fractions of time.

\begin{thmx} \label{thm:non-unif_syn}
Let $\{\cX^{\var}\}$ be a Markov perturbation of a synchronized DRN
$\cX^0$. Then, for every $\delta>0$ there exist
$\varepsilon_{\delta}>0$ and $E_{\delta} : \Om \to 2^{\N}$ such that
for $\mu$-a.e.\ $\om \in \Om$,
\begin{enumerate}
\item $E_{\delta}(\om)$ has large density, that is,
$$
\liminf \nolimits_{n\to \infty} \frac{\# (E_{\delta} (\om) \cap \{1,\ldots ,
  n\})}{n}>1 - \delta \, ;
$$
\item any two independent copies $\cX$, $\cY$ of $\cX^{\var}$ with
  $0\le \var < \var_{\delta}$ are very likely synchronized on
  $E_{\delta}(\om)$, that is,
$$
%\begin{equation}\label{1_degree_syn}
\P \bigl\{   X_n = Y_n  | X_0, Y_0 \in S \times \{\om\}\bigr\}>1 -
\delta \quad \forall n \in E_{\delta} (\om)\, .
$$
%\end{equation}
\end{enumerate}
Moreover, if $\cX^0$ is uniformly synchronized then $E_\de(\om)$ is co-finite.
\end{thmx}

It is not hard to see that uniform synchronization of $\cX^0$ is
equivalent to $P^0$ being uniformly contracting on the hyper-plane
$\Sigma_0 = \{v\in \R^k : \sum_{j=1}^k v_j = 0\}$, that is, with the
appropriate constants $a > 0$ and $0< \lambda < 1$,
$$
|P^0(n,\om) v| \le a \lambda^n |v| \quad \forall n \in \N_0 , v \in \Sigma_0
$$
for $\mu$-a.e.\ $\om \in \Om$. In fact, $P^0 (n,\om) \Sigma_0 = \{
0\}$ whenever $n\ge N(\om)$. This uniform contraction property is a special
form of uniform hyperbolicity, and as such is preserved under certain
perturbations. It is clear, however, that a generic cocycle will not
exhibit this property, and consequently a synchronized DRN will not in
general be uniformly synchronized. (Section 6 provides an
explicit example in this regard.) In the absence of uniform
synchronization, the proof of Theorem \ref{thm:non-unif_syn} crucially
depends on establishing connections between the probabilities of
synchronization on the one hand and the asymptotic behaviour of the
invariant distribution on the other hand. Once established, naturally
these connections can be strengthened under additional smoothness
assumptions. Specifically, say that a Markov perturbation
$\{\cX^{\var}\}$ is $C^m$ ($m\ge 1$) if $\var \mapsto P_{\cX^{\var}}
(1,\om)$ is $C^m$ on $[0,\var_0]$ for some $\var_0>0$ and $\mu$-a.e.\ $\om \in \Om$. For $C^m$-Markov
perturbations of a synchronized DRN this paper develops a Taylor
formula for the invariant distribution $p_{\var}$ of Theorem
\ref{thm:uniq_inv_distri}, thus refining (\ref{limit-distribution}):
With explicitly computable $\cF$-measurable functions $q^{(\ell)}: \Om \to \R^k$, $1\le
\ell\le m$, for $\mu$-a.e.\ $\om \in \Om$,
\begin{equation}\label{q_exp}
\lim\nolimits_{\var \to 0}  \frac{ p_{\var} (\om) - 
e_{J(\om)} - \sum\nolimits_{\ell = 1}^m \var^{\ell}q^{(\ell)}
(\om)/\ell !}{\var^m}  = 0 \, ;
\end{equation}
see Section \ref{sec4} for details. Utilizing (\ref{q_exp}), it is
possible to study quantitatively the scenarios of high-probability
synchronization as well as low-probability desynchronization. In
fact, these two scenarios coexist in an alternating
fashion along most realizations of the extrinsic noise. To formulate
this, the final main result of this work, consider the (possibly
empty) set where $p_{\var} - e_J$ is completely degenerate, that is, let
\begin{equation}\label{eq1n8a}
\Om_{\rm deg} = \{q^{(1)} = 0\} \cap \{q^{(2)} = 0\} \cap \ldots \cap
\{q^{(m)} = 0\}  \, .
\end{equation}
Except when $\mu (\Om_{\rm deg})=1$, where it is too degenerate to capture
desynchronization, (\ref{q_exp}) enables a quantitative version of the
coexistence claim as follows.

\begin{thmx}\label{thm:desyn}
Let $\{\cX^{\var}\}$ be a $C^m$-Markov perturbation $(m\ge 1)$ of a
synchronized DRN
$\cX^0$. Assume that $\mu
(\Om_{\rm deg})< 1$. Then, with the appropriate $0<a<1$ and $\ell \in \{1, \ldots
, m\}$, for every sufficiently small $\delta >0$ there exist
$\var_{\delta}>0, b_\de>0, c_{\de}>0$, and $E_{\delta}, F_{\delta}:\Om \to
2^{\N}$ such that for $\mu$-a.e.\ $\om \in \Om$,
\begin{enumerate}
\item $E_{\delta}(\om)$, $F_{\delta}(\om)$ are disjoint, have positive
  density, and together have 
  large density, that is,
$$
\liminf \nolimits_{n\to \infty} \frac{\# (E_{\delta} (\om) \cap \{1,\ldots ,
  n\})}{n}  > a - \delta \, , \quad
\liminf \nolimits_{n\to \infty} \frac{\# (F_{\delta} (\om) \cap \{1,\ldots ,
  n\})}{n} > 1 - a  - \delta
\, ;
$$
\item any two independent copies $\cX$, $\cY$ of $\cX^{\var}$ with
  $0\le \var < \var_{\delta}$ are very likely synchronized on
  $E_{\delta}(\om)$ but somewhat likely desynchronized on $F_{\delta}(\om)$, that is,
   \begin{align}
   {\P}\big\{  X_{n}=Y_{n}|X_0, Y_0\in S\times \{ \om\} \big\}&\ge 1- \var^{\ell}
   b_{\delta}   \quad \forall n \in E_{\delta} (\om) \, , \label{large_syn}\\
   {\P}\big\{ X_{n}\neq  Y_{n}|X_0, Y_0\in S \times \{\om\} \big\} &\ge \var^{\ell}
   c_{\delta } \quad \forall n \in F_{\delta} (\om) \, .
\label{small-desyn}
    \end{align}
\end{enumerate}
\end{thmx}

As it turns out, the important quantities $a$ and $\ell$ in Theorem
\ref{thm:desyn} can be expressed rather explicitly in terms of
$q^{(1)}, \ldots , q^{(m)}$. For instance, if $\mu (\Om_{\rm deg})>0$
then simply $a = \mu (\Om_{\rm deg})$, and $\ell$ equals the smallest
$1\le i \le m$ for which $\mu (\{q^{(1)} = 0\} \cap \ldots \cap
\{q^{(i)} =0\}) = \mu (\Om_{\rm deg})$; see Section 5 below for details.

Note that the (very mild) non-degeneracy assumption $\mu
(\Om_{\rm deg})<1$ is essential in Theorem \ref{thm:desyn}, as shown by the
trivial example $\cX^{\var} \equiv \cX^0$ for which (ii) fails. Informally put, Theorem \ref{thm:desyn} asserts that along almost every extrinsic noise
realization, for any two stochastic trajectories, high-probability synchronization \eqref{large_syn} and
low-probability desynchronization \eqref{small-desyn} occur along
different time subsequences in an alternating way (see also Figure
\ref{fig:multi-layer}), and moreover, the combined relative frequencies
of high-probability synchronization and low-probability desynchronization are arbitrarily close to 1. Thus the result rigorously confirms the speculation in \cite{J} mentioned at the outset.

\medskip

\begin{figure}[ht] 
\includegraphics[height=8cm]{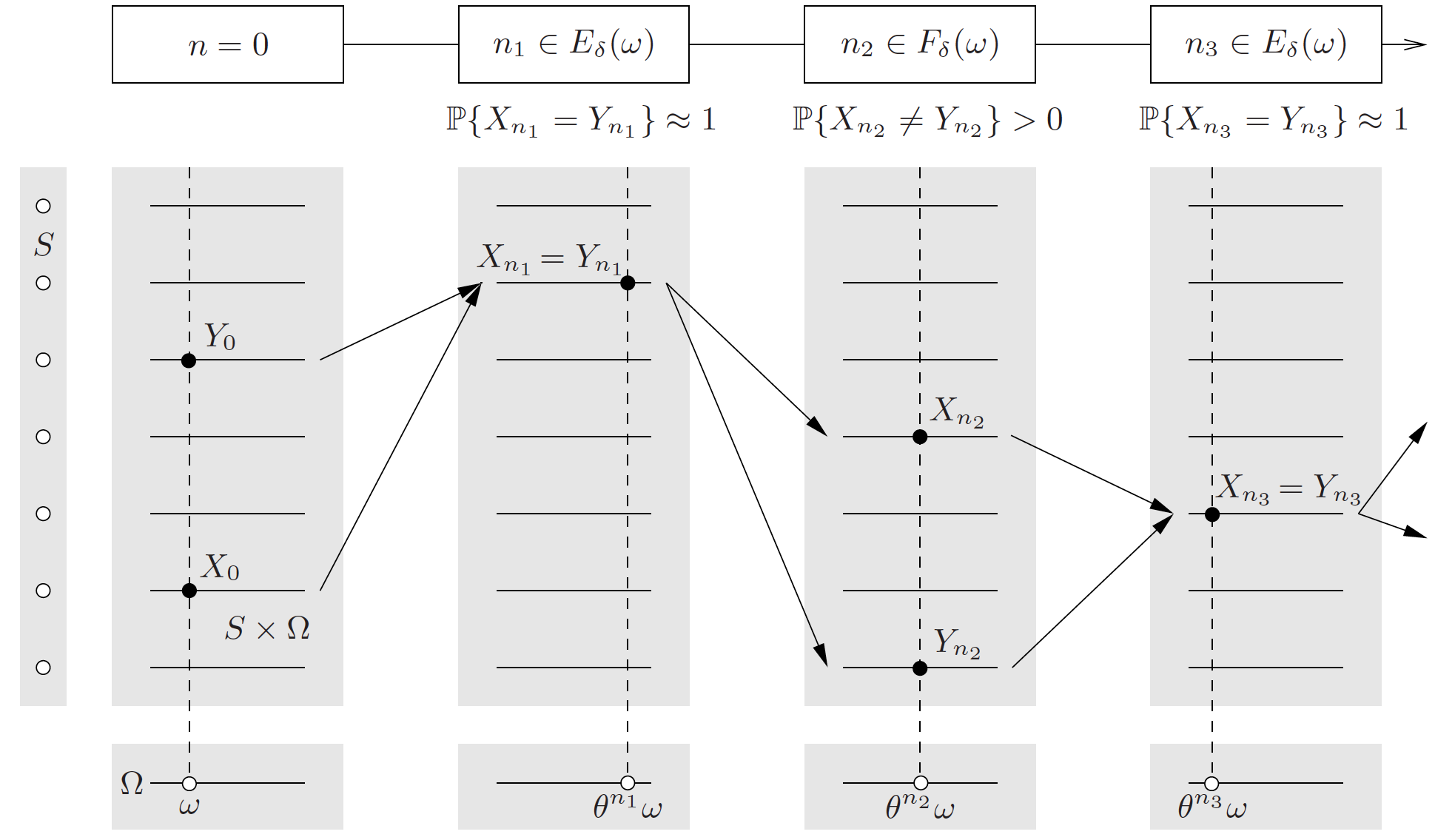}
\caption{By Theorem \ref{thm:desyn}, typical stochastic trajectories
  of a smooth Markov perturbation of a synchronized DRN exhibit
  alternating high-probability synchronization ($n\in E_{\de}$) and low-probability desynchronization ($n\in F_{\de}$).}\label{fig:multi-layer}
\end{figure}

%\noindent
This paper is organized as follows. Section 2 reviews
basic properties of MRN pertaining to the evolution of distributions
as well as to Lyapunov exponents. Section 3 proves Theorem
\ref{thm:uniq_inv_distri} by establishing the continuity  of Lyapunov
exponents for Markov perturbations of synchronized DRN, and also
proves Theorem \ref{thm:non-unif_syn} by linking synchronization to
invariant distributions. Section 4 studies the invariant distributions of smooth Markov perturbations, and derives the
asymptotic formula \eqref{q_exp}, through which finer properties of
$p_\var$ can be investigated. Theorem \ref{thm:desyn} is proved
in Section 5 utilizing (\ref{q_exp}). In Section \ref{sec6n} a concrete
Markov perturbation of a probabilistic Boolean network illustrates the
main concepts and results.

\section{Basic properties  of MRN}\label{sec2}

This section establishes a few properties of Markov random networks
that are instrumental in all that follows. To this end,
first a modicum of linear algebra notation and terminology is
reviewed.

Throughout, $k\ge 2$ is a positive integer, and $\K = \{1,\ldots, k\}$ for
convenience; as needed, endow $\K$ with the topology,
$\sigma$-algebra, and order inherited from $\R$. The canonical basis and identity map of
$\R^k$ are denoted $e_1, \ldots, e_k$ and $I_k$, respectively, and $|\,
\cdot \, |$ is the $\ell^1$-norm on $\R^k$, that is, $|v| = \sum_{j\in
  \K} |v_j|$ for every $v= \sum_{j\in \K} v_j e_j\in \R^k$; as usual,
$|\, \cdot \,|$ also denotes the induced norm on $\R^{k\times k}$,
that is, $|A| = \max_{|v|=1}|Av| = \max_{j\in \K} \sum_{i\in
  \K}|a_{i,j}|$ for every $A = (a_{i,j})\in \R^{k\times k}$. Given any
$V\subset \R^k$ (or $V\subset \R^{k\times k}$), $a\in \R$, $u\in
\R^k$ (or $u\in \R^{k\times k}$), and $A\in \R^{k\times k}$,
write
$$
aV = \{av:v\in V\}\, , \quad u+V = \{u+v: v \in V\}\, , \quad AV =
\{Av: v \in V\} \, .
$$
Also, for every $a\in \R$ let
$$
\Sigma_a = \left\{ v\in \R^k: \sum\nolimits_{j\in \K} v_j = a \right\}
\, .
$$
Plainly, $\Sigma_0$ is a $(k-1)$-dimensional linear subspace of
$\R^k$, and $\Sigma_a = ae_1 + \Sigma_0$; also, $\min_{v \in
  \Sigma_a}|v| = |a|$, and if $u\in \Sigma_a$,
$v\in \Sigma_b$ then $u\pm v \in \Sigma_{a\pm b}$. Furthermore,
consider
$$
\Sigma_a^+ = \left\{ v\in \Sigma_a : v_j \ge 0 \enspace \forall j \in
  \K \right\} \, ,
$$
and note that $\Sigma_a^+ = \varnothing$ if $a<0$, $\Sigma_0^+=\{0\}$,
and $\Sigma_a^+ = a\Sigma_1^+$ if $a>0$. In particular, $\Sigma_1^+$
may be identified with the set of all probability distributions on
$\K$ as well as the standard simplex in $\R^k$, that is, the convex hull of $\{e_1,
\ldots , e_k\}$.

For every $a\in \R$ let
$$
\cM_a = \left\{ A\in \R^{k\times k} : \sum\nolimits_{i\in \K} a_{i,j} =
a \enspace \forall j \in \K \right\} \, .
$$
Plainly, $\cM_0$ is a $(k^2-k)$-dimensional linear subspace of
$\R^{k\times k}$, and if $A\in \cM_a$, $B\in \cM_b$ then $A\pm B \in
\cM_{a\pm b}$ and $AB \in \cM_{ab}$. As is well known (and easy to check), the elements of
$\cM_a$ are characterized by simple invariance properties.

\begin{prop}\label{prop201}
Let $A\in \R^{k\times k}$. Then
\begin{enumerate}
\item $A\in \cM_0$ if and only if $A\R^k \subset \Sigma_0$;
\item $A\in \cM_a$ for some $a\in \R$ if and only if $A\Sigma_0
  \subset \Sigma_0$;
\item $A\in \cM_1$ if and only if $A\Sigma_a \subset \Sigma_a$ for
  some (and hence every) $a\ne 0$.
\end{enumerate}
\end{prop}

\noindent
Furthermore, consider
$$
\cM_a^+ = \{A \in \cM_a : a_{i,j} \ge 0 \enspace \forall i,j \in \K \}
\, ,
$$
and note that $\cM_a^+ = \varnothing$ if $a<0$, $\cM_0^+ = \{0\}$, and
$\cM_a^+ = a \cM_1^+$ if $a>0$. In particular, $\cM_1^+$ is the set of
all (column-)stochastic $k\times k$-matrices. A subclass of the latter
particularly relevant for this work are the stochastic $0$-$1$-matrices,
$$
\cM_{1, {\rm det}}^+ = \left\{ A\in \cM_1^+ : a_{i,j} \in \{0,1\}
  \enspace \forall i,j \in \K \right\} \, ,
$$
informally referred to as {\em deterministic\/} stochastic
matrices. Note that $\cM_1^+$ and $\cM_{1,{\rm det}}^+$ are closed under
matrix multiplication, just as $\cM_0$ and $\cM_1$ are. Not surprisingly, $\cM_1^+$ and $\cM_{1,{\rm det}}^+$ also
are characterized by simple invariance properties.

\begin{prop}\label{prop202}
Let $A\in \R^{k\times k}$. Then
\begin{enumerate}
\item $A\in \cM_1^+$ if and only if $A\Sigma_a^+ \subset \Sigma_a^+$
  for some (and hence every) $a>0$, and in this case $|A|=1$;
\item $A\in\cM_{1,{\rm det}}^+$ if and only if $A\{e_1, \ldots , e_k\}
  \subset \{e_1, \ldots , e_k\}$.
\end{enumerate}
\end{prop}

%\noindent
Recall from the Introduction that with every MRN one can associate the
transition cocycle $P_{\cX}$ and distribution $p_{\cX}$. The following properties of
$P_{\cX}$ and $p_{\cX}$ immediately follow from the definition of an
MRN; property (ii) justifies usage of the term ``cocycle''.

\begin{prop}\label{prop210}
Let $\cX$ be an MRN with transition cocycle $P_{\cX}$ and distribution
$p_{\cX}$. Then, for $\mu$-a.e.\ $\om \in \Om$,
\begin{enumerate}
\item $P_{\cX} (0,\om) = I_k$ and $P_{\cX} (n,\om) \in \cM_1^+$ for
  all $n\in \N_0$;
\item $P_{\cX}$ has the cocycle property
$$
%\begin{equation}\label{eqcocy}
P_{\cX} (m+n, \om) = P_{\cX} (m,\theta^n \om) P_{\cX} (n,\om) \quad
\forall m,n\in \N_0 \, ,
$$
%\end{equation}
and hence in particular
$$
P_{\cX} (n, \om) = P_{\cX} (1, \theta^{n-1}\om) P_{\cX} (1,
\theta^{n-2}\om) \cdots P_{\cX} (1,\om) \quad \forall n \in \N \, ;
$$
\item $p_{\cX} (n,\om) = P_{\cX} (n,\om) p_{\cX} (0,\om)\in
  \Sigma_1^+$ for all $n\in \N_0$.
\end{enumerate}
\end{prop}

\begin{remark}
(i) For the purpose of this paper, only the distributional structure of an
MRN $\cX$ matters. Whenever convenient, therefore, $P_{\cX}$ and
$p_{\cX}$ may be replaced by $2^{\N_0}\otimes \cF$-measurable
functions $P': \N_0 \times \Om \to \cM_1^+$ and
$p':\N_0 \times \Om \to \Sigma_1^+$, respectively, such
that
$$
\mu \left(  \bigcup\nolimits_{n\in \N_0} \bigl\{ P' (n, \, \cdot
  \, ) \ne P_{\cX} (n , \, \cdot \, ) \bigr\}\right) =
\mu \left(  \bigcup\nolimits_{n\in \N_0} \bigl\{ p' (n, \, \cdot
  \, ) \ne p_{\cX} (n , \, \cdot \, ) \bigr\}\right) = 0 \, ,
$$
and the assertions of Proposition \ref{prop210}, with $P', p'$
instead of $P_{\cX}, p_{\cX}$, hold for {\em all\/} $\om \in \Om$;
cf.\ \cite[sec.1.3.7]{arnold98book}.

(ii) One might call any $2^{\N_0}\otimes \cF$-measurable function
$P:\N_0 \times \Om \to \cM_1^+$ a {\em Markov cocycle\/} provided that
it has the cocycle property. Except for convenience, however, nothing
new is captured by this terminology: By Proposition \ref{prop210}, $P_{\cX}$
is a Markov cocycle for every MRN $\cX$, and conversely every Markov
cocycle, e.g., the cocycle $P'$ mentioned in (i), is the
transition cocycle of an MRN.
\end{remark}

%\noindent
To describe the long-time behaviour of an MRN $\cX$, say that $\cX$
{\em converges in distribution\/} if for $\mu$-a.e.\ $\om \in \Om$,
$$
%\begin{equation}\label{converge_in_distri}
        \lim\nolimits_{n\to \infty} \big|P_{\cX} (n,\omega) u -P_{\cX}
        (n,\omega)v \big|=0\quad \forall u, v \in \Sigma_1^+ \, ,
$$
%\end{equation}
or equivalently, $\lim_{n\to \infty} |P_{\cX} (n,\om) v|=0$ for all
$v\in \Sigma_0$. As the following lemma shows, in the special case of a DRN,
convergence in distribution is the same as synchronization. Note that
by Proposition \ref{prop202} an MRN $\cX$ is a
DRN precisely if $P_{\cX}(n,\om) \in \cM_{1,{\rm det}}^+$ for all
$n\in \N_0$ and $\mu$-a.e.\ $\om \in \Om$.

\begin{lemma}\label{lem:A-P_syn}
Let $\cX^0$ be a DRN. Then $\cX^0$ converges in distribution if and
only if $\cX^0$ is synchronized.
\end{lemma}

\begin{proof}
For the dtfs-RDS $T_{\cX^0}$ associated with $\cX^0$ by (\ref{deter_cocy-RDS}),
$$
T_{\cX^0} (n,\om)(s_j)=s_i \quad {\text {if and only if}} \quad P^0
(n,\om) e_j=e_i\, .
$$
Also, $P^0(n,\om) \in \cM_{1, {\rm det}}^+$, and so $|P^0(n,\om) e_i -
P^0(n,\om) e_j|\in \{0,2\}$.
Now, if $\cX^0$ is synchronized then for all $i,j\in \K$ and
$\mu$-a.e.\ $\om \in \Om$,
$$
P^0(n,\om) e_i = P^0 (n,\om)e_j \quad \forall n \ge N(\om) \, .
$$
Since every $v\in \Sigma_0$ is a linear combination of $\{e_i - e_j
:i,j \in \K\}$, also $P^0(n,\om) v = 0$ for all $n\ge
N(\om)$. Conversely, if $\cX^0$ converges in distribution then for
$\mu$-a.e.\ $\om \in \Om$ the set
$$
L(\om) := \{n\in \N_0 : P^0(n,\om) e_i = P^0 (n,\om) e_j \enspace \forall
 i,j\in \K \}
$$
is co-finite, and hence with $N(\om)= \inf L(\om)$ the DRN $\cX^0$ is synchronized.
\end{proof}

As will be seen next, even in the more general case of an arbitrary
MRN it is possible to characterize convergence in distribution. To this
end, given an MRN $\cX$ with transition cocycle $P_{\cX}$, call an
$\cF$-measurable function $p:\Om \to \Sigma_1^+$ an {\em invariant
  distribution\/} of $\cX$ if for $\mu$-a.e.\ $\om \in \Om$,
$$
P_{\cX} (n,\om) p(\om) = p (\theta^n \om) \quad \forall n \in \N_0 \, .
$$

\begin{lemma}\label{prop:syn_distr_equivalence}
Let $\cX$ be an MRN with transition cocycle $P_{\cX}$. Then the
following statements are equivalent:
\begin{enumerate}
\item $\cX$ converges in distribution;
\item there exists an invariant distribution $p$ of $\cX$ such that for every $q\in \Sigma_1^+$ and
 $\mu$-a.e.\ $\om \in \Om$,
$$
% \begin{equation}\label{pullback_q}
    \lim\nolimits_{n\to \infty} |P_{\cX} (n,\theta^{-n}\om) q  -
    p(\om)| = 0\
  $$
%  \end{equation}
(``$p$ is pull-back attracting'');
\item there exists an invariant distribution $p$ of $\cX$ such that
  for every $q\in \Sigma_1^+$ and
 $\mu$-a.e.\ $\om \in \Om$,
 \begin{equation}\label{pushforward_q}
    \lim\nolimits_{n\to \infty} |P_{\cX} (n, \om) q  -
    p(\theta^n \om)| = 0 
    \end{equation}
(``$p$ is forward attracting'').
\end{enumerate}
Moreover, the invariant distributions $p$ in {\rm (ii)} and {\rm (iii)} are uniquely determined and coincide
	$\mu$-a.e.\ on $\Om$.
\end{lemma}

\begin{proof}
To prove (i)$\Rightarrow$(ii), assume $\cX$ converges in
distribution. For $\mu$-a.e.\ $\om \in \Om$ and every $\ell \in \N$
pick $N_{\ell} (\om) \in \N$ such that
$$
|P_{\cX} (n,\om) e_i - P_{\cX} (n,\om) e_j| < \frac1{\ell} \quad
\forall i,j\in \K, n \ge N_{\ell} (\om) \, .
$$
Also pick $m_{\ell} \in \N$ with $m_{\ell} \ge \ell$ and $\mu (\{
N_{\ell} \le m_{\ell}\})>0$; assume w.l.o.g.\ that $(m_{\ell})_{\ell
  \in \N}$ is increasing. By Poincar\'e recurrence, it can be
assumed that $\theta^{-n}\om \in \{N_{\ell} \le m_{\ell}\}$ for
infinitely many $n$; in particular, pick $M_{\ell} (\om)> m_{\ell} +
M_{\ell-1} (\om) + M_{\ell - 1} (\theta^{-1} \om)$, where $M_0:= 0$, such that $\theta^{-M_{\ell} (\om)} \om \in \{N_{\ell} \le m_{\ell}\}$ for
$\mu$-a.e.\ $\om \in \Om$ and all $\ell \in \N$.
With this, consider the compact set
$$
C_\ell (\om) := P_{\cX} (M_{\ell} (\om), \theta^{-M_{\ell} (\om)}
\om )
\Sigma_1^+ \enspace \subset \enspace \Sigma_1^+ \, .
$$
Since $N_{\ell} (\theta^{-M_{\ell} (\om)}\om) \le m_{\ell}$ whereas
$M_{\ell} (\om)>N_{\ell}(\om)$, clearly $\mbox{\rm diam} \, C_\ell (\om) <
1/\ell$. Moreover, by the cocycle property, for every $n\ge M_{\ell}
(\om)$,
\begin{equation}\label{eqddel}
P_{\cX} (n, \theta^{-n} \om) \Sigma_1^+ = P_{\cX} (M_{\ell} (\om),
\theta^{-M_{\ell} (\om)}\om) P_{\cX} (n - M_{\ell}(\om),
\theta^{- n}\om ) \Sigma_1^+ \subset C_\ell (\om) \, ,
\end{equation}
and hence in particular $C_{\ell + 1} (\om) \subset C_\ell
(\om)$. It follows that for $\mu$-a.e.\ $\om\in \Om$ there exists a
unique $p(\om) \in \Sigma_1^+$ with $\{p(\om)\} = \bigcap_{\ell \in
  \N} C_\ell (\om)$. Clearly, $p:\Om \to \Sigma_1^+$ can be chosen to be
$\cF$-measurable. Furthermore, for $\mu$-a.e.\ $\om \in \Om$ and all
$\ell \in \N$, since $M_{\ell+ 1} (\theta \om) \ge M_{\ell} (\om) +
1$, 
$$
p(\theta \om) \in C_{\ell+1} (\theta \om) = P_{\cX} (1,\om)
P_{\cX} (M_{\ell+1} (\theta\om) -1, \theta^{1 - M_{\ell + 1}(\theta\om)}\om)
\Sigma_1^+  \subset P_{\cX}(1,\om) C_\ell (\om) \, .
$$
Since $\ell$ has been arbitrary, $p(\theta \om) = P_{\cX} (1, \om)
p (\om)$, that is, $p$ is an invariant distribution of $\cX$. Finally,
given $q\in\Sigma_1^+$, deduce from (\ref{eqddel}) that  $P_{\cX} (n,\theta^{-n}
\om) q \in C_\ell(\om)$ for $\mu$-a.e.\ $\om \in \Om$ and $n\ge
M_{\ell}(\om)$.  Thus
$$
|P_{\cX} (n ,\theta^{-n} \om) q - p(\om)| \le \frac2{\ell} \quad
\forall n \ge M_{\ell} (\om) \, ,
$$
and consequently $\lim_{n\to \infty} |P_{\cX} (n,\theta^{-n}\om) q - p(\om)|=0$.

To prove (ii)$\Rightarrow$(iii), let $p:\Om \to \Sigma_1^+$ be an
invariant distribution as in (ii). For $\mu$-a.e.\ $\om \in \Om$ and
every $\ell \in \N$ pick $N_{\ell}'(\om) \in \N$ such that
$$
|P_{\cX} (n, \theta^{-n} \om) e_j - p(\om)| < \frac1{\ell} \quad
\forall j \in \K, n \ge N_{\ell}'(\om) \, .
$$
Similarly to above, pick $m'_{\ell}\in \N$ with $m'_{\ell}\ge \ell$
and $\mu (\{ N'_{\ell} \le m'_{\ell}\})>0$ such that
$(m'_{\ell})_{\ell \in \N}$ is increasing. Again by Poincar\'{e}
recurrence, one can choose $M'_{\ell} (\om) \ge m'_{\ell}$ with
$\theta^{M'_{\ell} (\om)} \om \in \{N'_{\ell} \le m'_{\ell}\}$, and
consequently
$$
|P_{\cX} (M'_{\ell} (\om), \om) e_j - p(\theta^{M'_{\ell} (\om)}\om)| <
\frac1{\ell} \quad \forall j \in \K \, .
$$
For every $q\in\Sigma_1^+$ and $\mu$-a.e.\ $\om \in \Om$,
therefore, since $p$ is invariant,
\begin{align*}
|P_{\cX} (n,\om) q  - p(\theta^n \om)| & = \big|
P_{\cX} (n - M'_{\ell} (\om), \theta^{M'_{\ell} (\om)}\om ) \bigl( 
P_{\cX} (M'_{\ell}(\om), \om) q - p(\theta^{M'_{\ell} (\om)}\om )
\bigr)
\big|\\
& \le |P_{\cX} (M'_{\ell}(\om), \om) q - p(\theta^{M'_{\ell}
  (\om)}\om )| < \frac1{\ell} \, ,
\end{align*}
provided that $n\ge M'_{\ell} (\om)$. Since $\ell \in \N$ has been arbitrary,
this proves (\ref{pushforward_q}).

Finally, to see that (iii)$\Rightarrow$(i) simply choose $q=e_i$ and
$q=e_j$, respectively, and observe that for $\mu$-a.e.\ $\om \in \Om$,
$$
|P_{\cX} (n,\om) e_i - P_{\cX} (n,\om) e_j| \le |P_{\cX} (n , \om) e_i
- p (\theta^n \om)| + |P_{\cX} (n,\om) e_j - p(\theta^n \om)| \to 0
\quad \mbox{\rm as } n \to \infty \, ,
$$
showing that $\cX$ converges in distribution.

The assertion regarding the $\mu$-a.e.\ uniqueness of the invariant
distribution is obvious (and justifies usage of the common symbol $p$).
\end{proof}

Note that for a DRN $\cX^0$ every invariant distribution $p$ of
$\cX^0$ is concentrated on a single state, that is, $p(\om) =
e_{J(\om)}$ for a unique $J(\om)
\in \K$. Lemmas \ref{lem:A-P_syn} and \ref{prop:syn_distr_equivalence} together therefore have the
following corollary.

\begin{prop}\label{cor:A-syn}
Let $\cX^0$ be a synchronized DRN. Then there exist
$\cF$-measurable functions $J:\Om \to \K$ and $N^{\pm} : \Om \to \N$
such that for $\mu$-a.e.\ $\om \in \Om$,
$$
P^0(n,\om) e_{J(\om)} = e_{J(\theta^n \om)} \quad \forall n \in \N_0
\, ,
$$
and also, for every $j\in \K$,
\begin{align*}
P^0 (n,\om) e_j & = e_{J(\theta^n \om)} \quad \forall n \ge N^+ (\om)
\, , \\
P^0(n , \theta^{-n} \om) e_j & = e_{J(\om)} \enspace \enspace \quad  \forall n \ge N^-(\om)
\, .
\end{align*}
\end{prop}

Any function $J$ as in Proposition \ref{cor:A-syn} henceforth is
referred to as a {\em synchronization index\/} of the DRN $\cX^0$, and
$N^+$, $N^-$ are a {\em forward\/} and a {\em pull-back synchronization
time}, respectively. Note that $J$ is determined uniquely $\mu$-a.e.\
on $\Om$, whereas $N^{\pm}$ clearly are not. 
Also, it is not hard to see that
$N^\pm$ can be assumed constant $\mu$-a.e.\ on $\Om$ whenever
$\cX^0$ is uniformly synchronized.

\begin{remark}
As pointed out by the referee, it would be possible to {\em define}, for
$\mu$-a.e.\ $\om \in \Om$,
$$
N^+ (\om ) = \min\{n\in\N: P^0(n,\om)e_j=e_{J(\theta^n\om)}\enspace \forall
j\in\K\}\, ,
$$
and similarly for $N^-$. However, to emphasize that crucial parts of
this work, notably expression (\ref{eq4n0}) below, are independent of
the specific choice for (forward or pull-back) synchronization times,
no such definitions are adopted. Unless explicitly stated otherwise,
synchronization times $N^{\pm}$ can be arbitrary as long as they
comply with Proposition \ref{cor:A-syn}; see also Proposition \ref{prop6_1} below.
\end{remark}

The remainder of this section reviews a few pertinent facts regarding
Lyapunov exponents. Given an MRN $\cX$, for every $\om \in \Om$ and
$v\in \R^k$ let
\begin{equation}\label{LE}
\lambda_{\cX} (\om,
v)=\limsup\nolimits_{n\to\infty}\dfrac{\log|P_{\cX}(n,\om) v|}{n} \, ,
\end{equation}
with the convention that $\log 0 := - \infty$. As a consequence of the
classical Multiplicative Ergodic Theorem \cite{arnold98book,O68}, for $\mu$-a.e.\ $\om \in
\Om$ the $\limsup$ in (\ref{LE}) actually is a limit, and
$\lambda_{\cX} (\om , \, \cdot \, )$ attains at most $k$ different
real values which are constant $\mu$-a.e.\ on $\Om$. In fact, since $|P_{\cX}
(n,\om)|=1$ and $P_{\cX}
(n,\om) \Sigma_a \subset \Sigma_a$ for every $a$ by
Propositions \ref{prop201} and \ref{prop210}, a bit more can be said.

\begin{prop}\label{prop:LE-for-Mark-cocy}
Let $\cX$ be an MRN. Then, for $\mu$-a.e.\ $\om \in \Om$ and every
$v\in \R^k$,
\begin{equation}\label{LE=0}
 \lambda_{\cX}(\om,  v)\le0 \, ,
\end{equation}
and equality holds in {\rm (\ref{LE=0})} whenever $v \not \in
\Sigma_0$; in particular, $0$ is a Lyapunov exponent of $\cX$.
\end{prop}

In the special case of a DRN, the Lyapunov exponents $\lambda^0:=
\lambda_{\cX^0}$ can equal only $-\infty$ or $0$; as it turns out,
they also characterize synchronization.

\begin{prop}\cite[Thm.A]{HQWYY}\label{prop:A0-syn-LE}
Let $\cX^0$ be a DRN. Then $\cX^0$ is synchronized if and only if for
$\mu$-a.e.\ $\om \in \Om$,
$$
\lambda^0(\om, v)=\left\{
\begin{array}{ll}
- \infty & \mbox{\rm if } v \in \Sigma_0\, , \\
0 & \mbox{\rm if } v \not \in \Sigma_0 \, .
\end{array}
\right.
$$
\end{prop}

By noting that $e_i - e_j \in \Sigma_0$ for all $i,j\in \K$, a simple
sufficient condition for convergence in distribution of an MRN follows
immediately.

\begin{prop}\label{prop:simple-LE-to-syn}
Let $\cX$ be an MRN. If $\lambda_{\cX} (\om, v)<0$ for $\mu$-a.e.\
$\om \in \Om$ and all $v\in
\Sigma_0$ then $\cX$ converges in distribution.
\end{prop}

\begin{remark}
In \cite{arnold94positive}, a random Perron--Frobenius theorem is established for
positive (that is, strongly monotone) cocycles under $\log$-integrability
conditions, which in turn yields convergence in distribution. In
general, however, the transition cocycle $P_{\cX}$ of an MRN $\cX$ neither is positive (due
to possible zero entries) nor does it satisfy the $\log$-integrability conditions.
\end{remark}

\section{Convergence in distribution under Markov perturbations}

This section establishes two of the main results of this work,
Theorems \ref{thm:uniq_inv_distri} and \ref{thm:non-unif_syn}. The
proofs depend on three simple observations, of which the continuity
of Lyapunov exponents under Markov perturbations (Lemma
\ref{thm:markov-pert-simple-LE}) may be of independent
interest. Convergence in distribution, as well as continuity of the
invariant distribution, follow directly from the continuity of
Lyapunov exponents. First, observe that the cocycle property and an
induction argument immediately yield

\begin{prop}\label{lem:P-M}
    Let $\{\cX^{\var}\}$ be a Markov perturbation of a DRN $\cX^0.$ Then, for $\mu$-a.e. $\om\in\Om,$
 $$
%   \begin{equation}\label{P-M}
    |P_{\cX^\var}(n,\om)-P^0(n,\om)|\le n\var\quad\forall \var \ge 0 ,
    n\in \N_0 \, .
$$
%    \end{equation}
 \end{prop}

Next, given any synchronized DRN $\cX^0$, recall from Proposition
\ref{cor:A-syn} the notion of forward synchronization time $N^+ : \Om \to
\N$. Along a typical realization of the extrinsic noise the value of
$N^+$ reasonably often is not too large.

\begin{lemma}\label{lem:sequence}
Let $\cX^0$ be a synchronized DRN, and $N^+$ a forward synchronization
time. Let $m\in \N$ be such that $\mu (\{ N^+ \le m\}) >0$. Then, for
$\mu$-a.e.\ $\om\in\Om$ there exists a sequence $(n_{\ell})_{\ell \in
  \N}$ such that $n_{\ell + 1} - n_{\ell} \ge m$ and
$N^+(\theta^{n_{\ell}}\om )\le m$ for all $\ell \in \N$, as
well as $\limsup_{\ell \to \infty} \ell/n_{\ell} \ge \mu(\{ N^+ \le m\})/m$.
\end{lemma}

\begin{proof}
For $\mu$-a.e.\ $\om \in \Om$, the Birkhoff ergodic theorem yields
$$
\lim\nolimits_{n\to \infty} \frac{\# ( \{ k\in
  \N : N^+ (\theta^k\om) \le m\} \cap \{1, \ldots , n\}  )}{n} = \mu(\{ N^+ \le m\}) \, ,
$$
and hence for at least one $1\le i\le m$, possibly depending on $\om$,
$$
\limsup \nolimits_{n\to \infty} \frac{\# ( \{ k\in
  \N : N^+ (\theta^k\om) \le m\} \cap (i+m\N) \cap \{1, \ldots , n\} )}{n} \ge
\frac{\mu(\{ N^+ \le m\}) }{m} \, .
$$
In particular, the set $ \{ k\in  \N : N^+ (\theta^k\om) \le m\} \cap
(i+m\N)$ is infinite, and writing it as $\{n_{\ell} : \ell \in \N\}$
with $n_1<n_2 < \ldots$ yields a sequence $(n_{\ell})$ that has all the
asserted properties.
\end{proof}

The final preliminary observation establishes the continuity of all
Lyapunov exponents of MRN that are Markov perturbations of a
synchronized DRN. As pointed out in the Introduction, synchronization
of the unperturbed DRN is crucial for this result.

\begin{lemma}\label{thm:markov-pert-simple-LE}
Let  $\{\cX^{\var}\}$ be a Markov perturbation of a synchronized DRN
$\cX^0$. Then, for every sufficiently small $\var\ge 0$, for
 $\mu$-a.e.\ $\om\in\Om$ and every $v\in \R^k$, 
 the Lyapunov exponents of $\cX^\var$ satisfy
 \begin{equation}\label{LE_negative}
      \lambda_{\cX^{\var}} (\om, v) < 0 \enspace \mbox{if } v \in
      \Sigma_0 \, , \quad
 \lambda_{\cX^{\var}} (\om, v) = 0 \enspace \mbox{if } v \not \in
      \Sigma_0 \, .
\end{equation}
 Moreover, $\lambda_{\cX^{\var}}$ is continuous at $\var=0$, that is,
 $\lim_{\var\to 0} \lambda_{\cX^{\var}} (\om, v)=\lambda^0(\om, v)$.
\end{lemma}

\begin{proof}
By Propositions \ref{prop:LE-for-Mark-cocy} and \ref{prop:A0-syn-LE} all assertions are correct
in case $v\not \in \Sigma_0$, so henceforth assume $v\in
\Sigma_0$. Since $\cX^0$ is synchronized, pick $m\in \N$ with
$\mu(\{N^+ \le m\})>0$, and fix $\var < 1/m$. For $\mu$-a.e.\ $\om \in
\Om$, let $(n_{\ell})_{\ell \in \N}$ be as in Lemma
\ref{lem:sequence}. Then $P^0(m, \theta^{n_{\ell}}\om)v =0$, and since
$n_{\ell + 1} - n_{\ell} \ge m$,
\begin{equation}\label{Pv}
|P_{\cX^{\var}} (n_{\ell + 1} - n_{\ell}, \theta^{n_{\ell}} \om) v|
\le |P_{\cX^{\var}} (m, \theta^{n_{\ell}}\om)v| = \big|
\bigl( 
P_{\cX^{\var}} (m, \theta^{n_{\ell}}\om) - P^0(m,\theta^{n_{\ell}} \om)
\bigr) v
\big|  \le m \var |v| \, ,
\end{equation}
by Proposition \ref{lem:P-M}. Recall from Proposition \ref{prop201}
that $P_{\cX^{\var}} (n,\om) \Sigma_0 \subset \Sigma_0$ for all
$n$, so (\ref{Pv}) can be iterated,
\begin{align*}
|P_{\cX^{\var}} (n_{\ell}, \om) v| & = |P_{\cX^{\var}} (n_{\ell} -
n_{\ell -1}, \theta^{n_{\ell -1}}\om) P_{\cX^{\var}} (n_{\ell -1} -
n_{\ell -2}, \theta^{n_{\ell -2}}\om) \cdots P_{\cX^{\var}} (n_{2} -
n_{1}, \theta^{n_{1}}\om) P_{\cX^{\var}} ( n ,\om) v| \\
& \le (m\var)^{\ell -1} |v| \, ,
\end{align*}
which in turn yields, utilizing Lemma \ref{lem:sequence},
\begin{align*}
\lambda_{\cX^{\var}} (\om , v) & = \lim\nolimits_{n\to \infty}
\frac{\log |P_{\cX^{\var}} (n,\om)v|}{n} \le \liminf\nolimits_{\ell
  \to \infty} \frac{(\ell -1) \log (m\var) + \log |v|}{n_{\ell}} \\
& = \log (m\var) \limsup\nolimits_{\ell \to \infty} \frac{\ell
  -1}{n_{\ell}} \le \frac{\log (m\var) \mu (\{N^+ \le m\})}{m} < 0 \, .
\end{align*}
Thus (\ref{LE_negative}) holds, and letting $\var \to 0$ shows that $\lim_{\var \to 0}
\lambda_{\cX^{\var}} (\om, v) = -\infty$ whenever $v\in \Sigma_0$.
Proposition \ref{prop:A0-syn-LE} then concludes the proof.
\end{proof}

The scene is now set for a rather straightforward 

\begin{proof}[Proof of Theorem \ref{thm:uniq_inv_distri}]
To prove the existence of an invariant distribution, as well as (i) and (ii),
simply observe that Lemma \ref{thm:markov-pert-simple-LE} and
Proposition \ref{prop:simple-LE-to-syn} together imply that $\cX^{\var}$ converges in
distribution for all sufficiently small $\var > 0$. By Lemma
\ref{prop:syn_distr_equivalence}, there exists
a unique invariant distribution of $\cX^{\var}$, denoted $p_{\var}$, that is both
pull-back and forward attracting. 

To establish (iii), that is, the continuity of $\var \mapsto p_{\var}$ at
$\var = 0$, recall from Proposition \ref{cor:A-syn} that $P^0 (N^-(\om),
\theta^{-N^-(\om)} \om) v = e_{J (\om)}$ for $\mu$-a.e.\ $\om \in \Om$
and every $v\in \Sigma_1^+$, where $N^-$ denotes a pull-back
synchronization time. By the invariance of $p_{\var}$ and Proposition
\ref{lem:P-M},
\begin{align}\label{eq310}
|p_{\var} (\om) - e_{J(\om)}| & = |P_{\cX^{\var}} (N^-(\om),
\theta^{-N^-(\om)}\om) p_{\var} (\theta^{-N^-(\om)}\om) - e_{J(\om)}|
\nonumber \\
& = \big|
\bigl( P_{\cX^{\var}} (N^-(\om), \theta^{-N^-(\om)}\om) - P^0  (N^-(\om), \theta^{-N^-(\om)}\om) 
\bigr)
p_{\var} (\theta^{-N^-(\om)}\om)
\big| \nonumber \\
& \le N^-(\om) \var \, ,
\end{align}
and hence clearly $\lim_{\var\to 0} p_{\var} (\om) = e_{J(\om)}$ for
$\mu$-a.e.\ $\om \in \Om$.

Finally, if $\cX^0$ is uniformly synchronized then $N^-(\om)\le N$ for
$\mu$-a.e.\ $\om \in \Om$ and the appropriate $N\in \N$, so
(\ref{eq310}) shows that the convergence in (\ref{limit-distribution}) is uniform on a
set of full $\mu$-measure.
\end{proof}

Note that, informally put, Theorem \ref{thm:uniq_inv_distri} provides a description, at
the level of (invariant) distributions, of what synchronization looks
like for a sufficiently small Markov perturbation of a synchronized
DRN. In order to rephrase this description at the level of stochastic
trajectories (Theorem \ref{thm:non-unif_syn}), the following simple
linear algebra observation is helpful.

\begin{prop}\label{lem:vj}
Let $v\in \Sigma_1^+$. Then $v_j = 1 - |v - e_j|/2$ for every $j\in \K$.
\end{prop}

\begin{proof}[Proof of Theorem  \ref{thm:non-unif_syn}]
Fix $\var_0>0$ so small that all conclusions in Theorem \ref{thm:uniq_inv_distri} hold
whenever $0\le \var < \var_0$. In particular, therefore, given
$\delta>0$, for $\mu$-a.e.\ $\om \in \Om$ there exists $N(\om)\in \N$
and $f(\om , \delta)>0$ such that
$$
|p_{\cX^{\var}} (n,\om) - p_{\var} (\theta^n \om)| < \frac{\delta}{2}
\enspace \forall n \ge N(\om) \quad \mbox{and} \quad
|p_{\var} (\om) - e_{J(\om)}| < \frac{\delta}{2} \enspace \forall \var
< f(\om , \delta) \, .
$$
Pick $0< \var_{\delta} \le \var_0$ so small that $\mu (\{ f(\, \cdot
\, , \delta) \ge \var_{\delta}\})> 1 - \delta$. Letting $\Om_{\delta}
= \{ f(\, \cdot \, , \delta)\ge \var_{\delta}\}$ and $E_{\delta} (\om)
= \{n\ge N(\om): \theta^n \om \in \Om_{\delta}\}$ for convenience, by
the Birkhoff ergodic theorem,
$$
\lim\nolimits_{n\to \infty} \frac{\# (  E_{\delta}(\om) \cap \{1, \ldots , n\} )}{n} = \mu (\Om_{\delta}) > 1 - \delta \, ,
$$
which proves (i). Furthermore, if $0<\var< \var_{\delta}$ and $n\in
E_{\delta} (\om)$ then $|p_{\cX^{\var}} (n,\om) - e_{J(\theta^n
  \om)}|<\delta $ by the triangle inequality, and
\begin{align*}
 {\P}\big\{X_{n}  =(s_{J(\theta^{n}\om)},\theta^{n}\om)|X_0\in
 S\times\{\om\}\big\}& = p_{\cX^{\var}}(n,\om)_{J(\theta^n \om)}
 = 1 - \frac{|p_{\cX^{\var}} (n,\om) - e_{J(\theta^n \om)}|}{2} \\
&  > 1- \frac{\delta}{2}
\, ,
\end{align*}
where the second equality is due to Proposition \ref{lem:vj}. Thus, if
$\cX, \cY$ are two independent copies of $\cX^{\var}$, with
$0<\var<\var_{\delta}$, then for all $n\in E_{\delta}(\om)$,
\begin{align*}
{\P}\big\{X_{n}=Y_{n}|X_0,Y_0\in S\times\{\om\}\big\} & \ge
{\P}\big\{X_{n}=Y_{n} = (s_{J(\theta^n \om)}, \theta^n \om)
|X_0,Y_0\in S\times\{\om\}\big\} > \left(1-\dfrac{\de}{2}\right)^2\\
& >1-\de \,  ,
\end{align*}
which proves (ii).

Finally, if $\cX^0$ is uniformly synchronized then, by Theorem
\ref{thm:uniq_inv_distri} it can be assumed that $f(\om,\de)$ is
independent of $\om$. But then $\mu (\Om_{\de})=1$, indeed $\mu \left(
  \bigcap_{n\ge 0} \theta^{-n} \Om_{\de}\right)=1$, and so for
$\mu$-a.e.\ $\om \in \Om$ the set $E_{\de} (\om) = \{n \ge N(\om)\}$
is co-finite.
\end{proof}

\section{Invariant distributions for smooth Markov perturbations}\label{sec4}

This section establishes an asymptotic expansion for the invariant
distribution of a $C^m$-Markov perturbation
$\{\cX^{\var}\}$, with $m\ge 1$, of a synchronized DRN $\cX^0$. Remember that due to
the ``deterministic'' nature of $\cX^0$, at all times each state leads to precisely one subsequent state. By contrast, due to
the presence of intrinsic noise, states of $\cX^{\var}$ typically have
a small but positive probability of leading to more than one
subsequent state. As may be expected, these deviations can be described in
terms of the derivatives of $\var \mapsto P_{\cX^{\var}}$ at $\var =
0$. To do this explicitly, for every $0\le \ell \le m$ let
$$
P_{\cX^{\var}}^{(\ell)} (n,\om) = \left. \frac{{\rm d}^{\ell}}{{\rm d}
  \var^{\ell}} P_{\cX^{\var}} (n,\om) \right|_{\var = 0} \, ;
$$
note that $P_{\cX^{\var}}^{(0)} (n,\om )= P^0(n,\om) \in \cM_{1, {\rm
    det}}^+$, and clearly $P_{\cX^{\var}}^{(\ell)}$ is $2^{\N_0}
\otimes \cF$-measurable, with $P_{\cX^{\var}}^{(\ell)} (n,\om)\in
\cM_0$ for all $1\le \ell \le m$, $n\in \N_0$, and $\mu$-a.e.\ $\om
\in \Om$. Moreover, for every $n\in \N_0$ and $\mu$-a.e.\ $\om \in
\Om$,
$$
\lim\nolimits_{\var \to 0}  \frac{P_{\cX^{\var}} (n,\om) - P^0(n,\om)
  - \sum_{\ell = 1}^m \var^{\ell} P_{\cX^{\var}} ^{(\ell)} (n,\om)/\ell !}{\var^m} = 0 \, .
$$
Throughout this section, assume that the DRN $\cX^0$ is synchronized. By
Proposition \ref{cor:A-syn}, $\cX^0$ admits an (essentially unique)
synchronization index $J$ as well as a pull-back synchronization time
$N^-$. The following simple observation will be used several times
below.

\begin{prop}\label{prop4n1}
Let $\cX^0$ be a synchronized DRN, and $N^-$ a pull-back
synchronization time. Then, for $\mu$-a.e.\ $\om \in \Om$,
$$
P^0(N^-(\om), \theta^{-N^-(\om)} \om) v = 0 \quad \forall v \in
\Sigma_0 \, .
$$
\end{prop}

Utilizing $J$ and $N^-$, as well as $P_{\cX^{\var}}^{(\ell)}$, define $\cF$-measurable functions $q^{(0)},
q^{(1)}, \ldots , q^{(m)}:\Om \to \R^k$ as
$$
q^{(0)} (\om) = e_{J(\om)} \, ,
$$
and inductively for $1\le \ell \le m$,
\begin{equation}\label{eq4n0}
q^{(\ell)} (\om) = \sum\nolimits_{i=0}^{\ell-1} \left( \!\!
\begin{array}{c} \ell \\ i \end{array}\!\!
\right) P_{\cX^{\var}}^{(\ell - i)} (N^-(\om), \theta^{-N^-(\om)}\om)
q^{(i)} (\theta^{-N^-(\om)}\om) \, .
\end{equation}
Note that $q^{(0)}\in \Sigma_1^+$, whereas clearly $q^{(\ell)}(\om)
\in \Sigma_0$ for all $1\le \ell \le m$ and $\mu$-a.e.\ $\om \in \Om$;
in particular,
$$
q^{(1)}(\om) = P_{\cX^{\var}} ^{(1)} (N^-(\om), \theta^{-N^-(\om)}\om)
e_{J(\theta^{-N^-(\om)}\om)} \, .
$$
The main result of this section, Lemma \ref{lem4n2} below, provides a
Taylor formula approximately expressing the invariant distribution of
$\cX^{\var}$ in terms of the quantities $q^{(\ell)}$, $1\le \ell\le
m$. As the attentive reader will have noticed, by definition these
quantities depend on the choice of $N^-$. However, it is readily
seen that choosing a different $N^-$ yields the same $q^{(1)},
\ldots, q^{(m)}$ for $\mu$-a.e.\ $\om \in \Om$.

\begin{lemma}\label{lem4n2}
Let $\{\cX^{\var}\}$ be a $C^m$-Markov perturbation $(m\ge 1)$ of a
synchronized DRN $\cX^0$. For every
sufficiently small $\var\ge 0$ let $p_{\var}$ be the invariant distribution
of $\cX^{\var}$, as in Theorem \ref{thm:uniq_inv_distri}. Then, for $\mu$-a.e.\ $\om \in
\Om$,
\begin{equation}\label{eq4n3}
\lim\nolimits_{\var \to 0}  \frac{p_{\var} (\om) - e_{J(\om)} -
  \sum_{\ell = 1}^m \var^{\ell} q^{(\ell)} (\om)/\ell
  !}{\var^m} = 0 \, .
\end{equation}
\end{lemma}

\begin{proof}
With $N^-$ denoting any pull-back synchronization time of $\cX^0$,
write $\om^- := \theta^{-N^-(\om)}\om$ for convenience. To establish
(\ref{eq4n3}), it will be shown that for every $0\le k \le m$ and
$\mu$-a.e.\ $\om \in \Om$,
\begin{equation}\label{eq4n4}
\lim\nolimits_{\var \to 0}  \frac{p_{\var} (\om) - e_{J(\om)} -
  \sum_{\ell = 1}^k \var^{\ell} q^{(\ell)} (\om)/\ell
  !}{\var^k} = 0 \, .
\end{equation}
By Theorem \ref{thm:uniq_inv_distri}(ii), clearly (\ref{eq4n4}) is
correct for $k=0$, with an ``empty'' sum understood to equal $0$, as
usual. Assume that (\ref{eq4n4}) is correct for some $0\le k < m$. Then
$$
P_{\cX^{\var}}(n,\om) = P^0(n,\om) + \sum\nolimits_{\ell = 1}^k
\frac{\var^{\ell}}{\ell ! } P_{\cX^{\var}}^{(\ell)} (n,\om) +
\frac{\var^{k+1}}{(k+1)!} P_{\cX^{\var}}^{(k+1)} (n,\om) +\var^{k+1}
R_{\var}(n,\om) \, ,
$$
with the appropriate $R_{\var} (n,\om)\in \cM_0$ and $\lim_{\var \to
  0} R_{\var} (n,\om) = 0$, as well as
$$
p_{\var} (\om) = e_{J(\om)} + \sum\nolimits_{\ell = 1}^k
\frac{\var^{\ell}}{\ell !} q^{(\ell)} (\om) + \var^k r_{\var} (\om)
\, ,
$$
with the appropriate $r_{\var}(\om)\in \Sigma_0$ and $\lim_{\var \to
0} r_{\var} (\om) = 0$. Thus, by the invariance of $p_{\var}$ and
$e_J$, Proposition \ref{prop4n1}, and the definition of $q^{(\ell)}$,
\begin{align*}
p_{\var} (\om) - e_{J(\om)} & = P_{\cX^{\var}} (N^-(\om), \om^-)
p_{\var} (\om^-) - e_{J(\om)} \\
& = \bigl( P_{\cX^{\var}} (N^-(\om), \om^-) - P^0
(N^-(\om), \om^-)\bigr) p_{\var} (\om^-) \\
& = \left( 
 \sum\nolimits_{\ell = 1}^k
\frac{\var^{\ell}}{\ell ! } P_{\cX^{\var}}^{(\ell)} (N^-(\om),\om^-) +
\frac{\var^{k+1}}{(k+1)!} P_{\cX^{\var}}^{(k+1)} (N^-(\om),\om^-) +\var^{k+1}
R_{\var}(N^-(\om),\om^-) 
\right) \\
& \qquad 
\cdot \left(
e_{J(\om^-)} + \sum\nolimits_{\ell = 1}^k
\frac{\var^{\ell}}{\ell !} q^{(\ell)} (\om^-) + \var^k r_{\var} (\om^-)
\right) \\
& = \sum\nolimits_{\ell = 1}^k \sum\nolimits_{i=0}^{\ell -1}
\frac{\var^{\ell}}{i! (\ell -i)!}  P_{\cX^{\var}}^{(\ell - i)}
(N^-(\om), \om^-)q^{(i)} (\om^-) \\
& \qquad + \var^{k+1} \sum\nolimits_{i=0}^k
\frac1{i! (k+1 - i)!} P_{\cX^{\var}}^{(k+1 - i)} (N^-(\om), \om^-)
q^{(i)}(\om^-) + \var^{k+1} R_{\var}'(n,\om) \\
& = \sum\nolimits_{\ell =1}^{k+1} \frac{\var^{\ell}}{\ell !}
q^{(\ell)}(\om) + \var^{k+1} R_{\var}'(n,\om) \, ,
\end{align*}
with the appropriate $R_{\var}'(n,\om)\in \cM_0$ and $\lim_{\var\to 0}
R_{\var}'(n,\om) = 0$. Thus (\ref{eq4n4}) is correct with $k$ replaced
by $k+1$, and induction establishes (\ref{eq4n3}).
\end{proof}

\begin{remark}
For a homogeneous Markov chain, a series expansion (or updating
formula) for invariant distributions under (regular as well as singular) perturbations has been
derived in \cite{Altman}. In this context, the transition cocycle
$P_{\cX^{\var}}$ may be viewed as a ``random"
version of a Markov chain, with the Markov perturbation  as a regular
perturbation. Accordingly, Lemma \ref{lem4n2} may be viewed as
a random version of the series expansion in \cite{Altman} under regular perturbations.
\end{remark}

The remainder of this section develops a simple condition guaranteeing
that the invariant distribution $p_{\var}$ of a Markov perturbation
$\{\cX^{\var}\}$ is degenerate at $\var =0$ in that $q^{(1)}=0$ in
(\ref{eq4n3}). In this, a crucial role is played by the quantity
$$
d(\om):= P_{\cX^{\var}}^{(1)} (1,\om) e_{J(\om)} \, ,
$$
which may be thought of as the first-order (one-step) {\em probability
dissipation\/} under the Markov perturbation. Notice that $d$ is
$\cF$-measurable, and $d(\om)\in \Sigma_0$ for $\mu$-a.e.\ $\om\in
\Om$. Moreover, except for the $J(\theta \om)$-th component, all
components of $d(\om)$ are non-negative, that is, $d(\om)_j\ge 0$ for
every $j\in \K \setminus \{J(\theta \om)\}$, and hence $d(\om)_{J(\theta
  \om)}\le 0$. From the cocycle properties of $P_{\cX^{\var}}$ and
$P^0$, it is readily deduced that for all $n\in \N$ and $\mu$-a.e.\ $\om\in \Om$,
\begin{equation}\label{eq4n4n}
P_{\cX^{\var}}^{(1)}(n,\om) = \sum\nolimits_{\ell =0}^{n-1}
P^0(n-1-\ell ,
\theta^{\ell +1}\om) P_{\cX^{\var}}^{(1)} (1,\theta^{\ell}\om)
P^0(\ell ,\om) \, ,
\end{equation}
and consequently $q^{(1)}$ can be written neatly in terms of $d$,
thus revealing its pull-back nature,
\begin{align}\label{eq4n5}
q^{(1)}(\om) & = \sum\nolimits_{\ell =0}^{N^-(\om) -1} P^0(N^-(\om) - 1 - \ell,
\theta^{1+\ell- N^-(\om)} \om) d(\theta^{\ell - N^-(\om)} \om )
\nonumber \\
& = \sum\nolimits_{\ell = 0}^{N^-(\om) -1} P^0(\ell, \theta^{-\ell} \om)
d(\theta^{-\ell - 1} \om)\, .
\end{align}
In order to trace the dissipation of probability, for $\mu$-a.e.\ $\om
\in \Om$ consider the subset $S_{\bullet}$ of $S$ given by
$$
S_{\bullet} (\om ) = \{ s_j \in S : P^0(1,\om) e_j = e_{J(\theta \om )}\}
\, .
$$
Thus $S_{\bullet }(\om)$ contains precisely those states that lead to the
synchronized state $s_J$ in a single step. By invariance, $s_{J(\om)} \in
S_{\bullet} (\om)$. Also, consider the subset of $\Om$ given by
$$
\Om_{\bullet} = \{ \om \in \Om : d(\om)_j = 0 \enspace \forall s_j
\not \in S_{\bullet} (\theta \om )\} \, \in \, \cF \, .
$$
Thus, for $\om \in \Om_{\bullet}$ first-order probability
dissipation occurs only to states that immediately lead to the
synchronized state; see also Figure \ref{fig:prob_diss}. Intuitively, it is
plausible that if first-order probability dissipation is thus
restricted for all times up to the (finite) synchronization time of
$\cX^0$ then $p_{\var}$ differs from the single-state invariant
distribution $e_J$ of $\cX^0$ only by higher orders of
$\var$, that is, $q^{(1)}=0$.

\begin{lemma}\label{lem4n4}
Let $\{\cX^{\var}\}$ be a $C^m$-Markov perturbation $(m\ge 1)$ of a
synchronized DRN $\cX^0$, with pull-back
synchronization time $N^-$. Then, for $\mu$-a.e.\ $\om \in \Om$, the
following hold:
\begin{enumerate}
\item If $d(\theta^{-1}\om)\ne 0$ then $q^{(1)}(\om) \ne 0$;
\item If $d(\theta^{-1}\om)= 0$ and $\om \in \bigcap_{\ell
    =1}^{N^-({\om}) -1} \theta^{\ell + 1} \Om_{\bullet}$ then $q^{(1)}
  (\om)=0$.
\end{enumerate}
\end{lemma}

\begin{proof}
Notice first that each of the $N^-(\om)$ terms in (\ref{eq4n5}) lies in
$\Sigma_0$, with all components non-negative, except possibly for the
$J(\omega)$-th component. Thus $q^{(1)}=0$ if and only if $P^0(\ell ,
\theta^{-\ell} \om) d(\theta^{-\ell -1}\om)=0$ for all $\ell = 0, 1,
\ldots, N^-(\om)-1$. With this, on the one hand, if
$d(\theta^{-1}\om)\ne 0$ then $q^{(1)} (\om)\ne 0$, which proves (i). On the
other hand, if $\om \in \theta^{\ell + 1} \Om_{\bullet}$ for some
$1\le \ell \le N^-(\om)-1$ then $P^0(1, \theta^{-\ell}\om)
d(\theta^{-\ell -1} \om)$ lies in $\Sigma_0$, with all components
non-negative, except for the $J(\theta^{1-\ell} \om)$-th component,
for which
\begin{align*}
P^0(1, \theta^{-\ell} \om) d(\theta^{-\ell -1}\om)_{J(\theta^{1-\ell}
  \om)} & = \sum\nolimits_{j=1}^k  P^0(1, \theta^{-\ell}
\om)_{J(\theta^{1-\ell}\om), j} d(\theta^{-\ell - 1} \om)_j \\
& = \sum\nolimits_{j\in S_{\bullet}(\theta^{-\ell} \om)}  P^0(1, \theta^{-\ell}
\om)_{J(\theta^{1-\ell}\om), j} d(\theta^{-\ell - 1} \om)_j \\
& = \sum\nolimits_{j\in S_{\bullet}(\theta^{-\ell} \om)}
d(\theta^{-\ell - 1} \om)_j   =  \sum\nolimits_{j=1}^k
d(\theta^{-\ell - 1} \om)_j  = 0 \, ,
\end{align*}
where the second and fourth equality are due to $\om \in \theta^{\ell + 1} \Om_{\bullet}$. Thus $P^0(1,
\theta^{-\ell} \om) d(\theta^{-\ell -1}\om)=0$, and by the cocycle
property $P^0(\ell, \theta^{-\ell} \om) d(\theta^{-\ell -1}\om)=0$ as
well, that is, $q^{(1)}(\om)=0$, which proves (ii).
\end{proof}

\begin{remark}
In a spirit similar to Lemma \ref{lem4n4}(ii), one might derive conditions for
higher-order degeneracy of $p_{\var}$ at $\var =0$, that is, for
$q^{(1)} = \ldots = q^{(\ell)} = 0$ for some $2\le \ell \le
m$. However, since the pertinent analogues of (\ref{eq4n0}) and (\ref{eq4n5}) are
considerably more cumbersome in this case, such conditions likely are
of limited practical use.
\end{remark}

\begin{figure}[ht] 
\includegraphics[height=7.4cm]{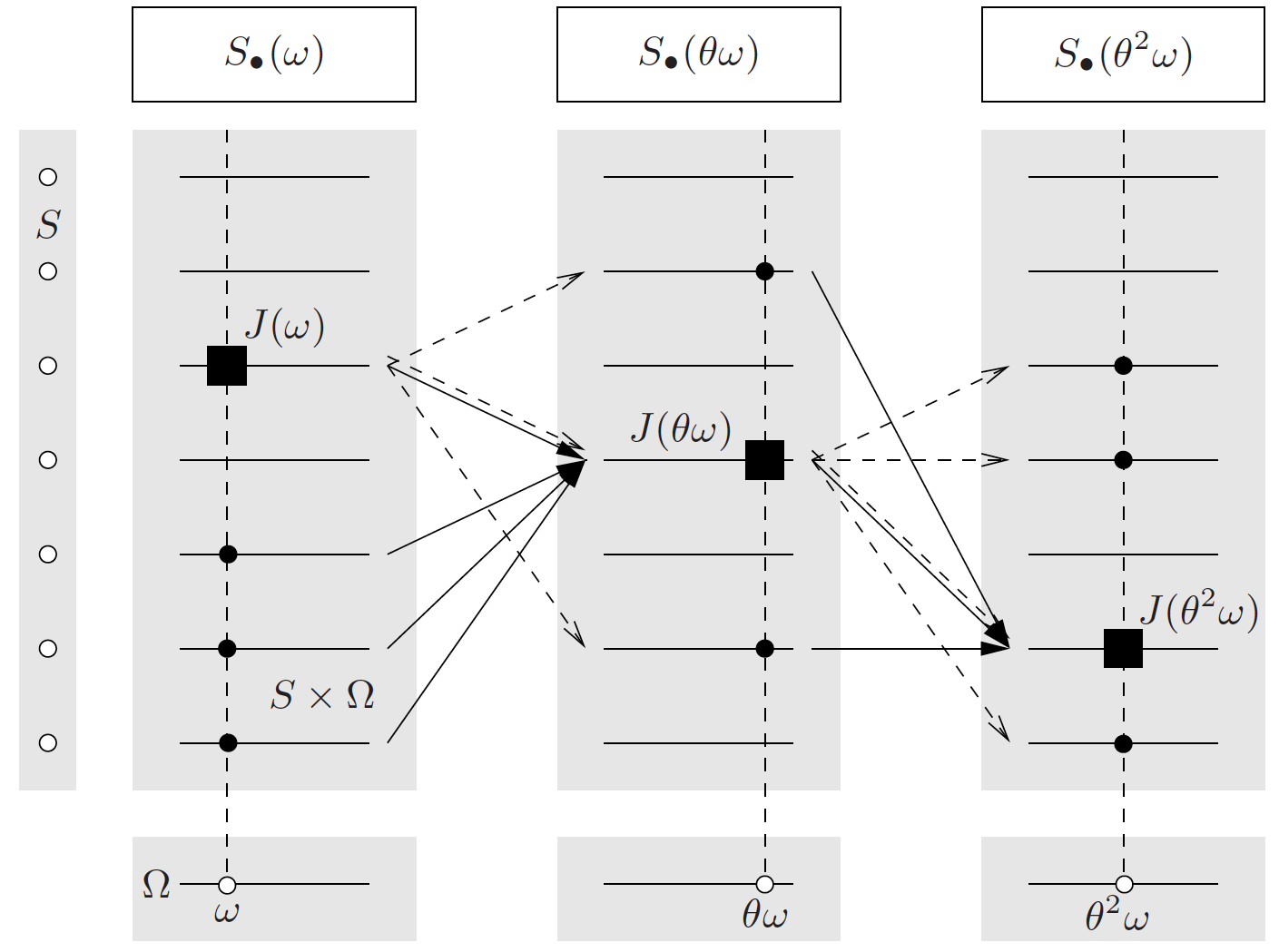}
\caption{For $\om \in \Om_{\bullet}$, first-order probability
  dissipation (dashed arrows) under a smooth Markov perturbation occurs only to states that immediately
  lead to the synchronized state (black squares) of the synchronized DRN
  (solid arrows).}\label{fig:prob_diss}
\end{figure}

\section{Alternating patterns of synchronization and desynchronization}\label{sec5}

This section is  devoted to the proof of Theorem \ref{thm:desyn}. As
with Theorem \ref{thm:non-unif_syn}, the main idea is to link the
behaviour of the invariant distribution $p_{\var}$ as $\var \to 0$ to
synchronization at the level of trajectories. Utilizing the asymptotic
expansion (\ref{eq4n3}) enables a refinement of Theorem
\ref{thm:non-unif_syn}, via quantitative descriptions of
high-probability synchronization as well as low-probability
desynchronization. For the sake of clarity, these two scenarios are
first addressed in two separate lemmas; a combination of both results
then yields Theorem \ref{thm:desyn}. Throughout, assume that $\{\cX^{\var}\}$ is a $C^m$-Markov
perturbation $(m\ge 1)$ of a syn\-chro\-nized DRN $\cX^0$. For
convenience, for every $1\le \ell \le m$ let
\begin{eqnarray}\label{def:Om_ell}
\Om^{(\ell)} = \{q^{(1)}=0\} \cap \ldots \cap \{q^{(\ell)} = 0\} \in
\cF \, ;
\end{eqnarray}
note that $\Om^{(0)}:= \Om \supset \Om^{(1)} \supset \ldots \supset
\Om^{(m)} = \Om_{\rm deg}$, with $\Om_{\rm deg}$ considered already in (\ref{eq1n8a}).

\begin{lemma}\label{thm:large-Syn}
Assume that $\mu \bigl( \Om^{(\ell)} \bigr)>0$ for some $1\le \ell \le m$. Then,
for every $0<\delta<\mu \bigl( \Om^{(\ell)} \bigr)$ there exist
$\var_\de> 0$, $b_{\de} :[0, \var_{\de})\to \R$ with $b_{\de}
(\var)\ge 0$ and $\lim_{\var\to 0}b_\de(\var)=0$,
and $E_\delta:\Om\to2^{\N}$ such that for $\mu$-a.e.\ $\om\in\Om$,
\smallskip
\begin{enumerate}
\item 
$\liminf\nolimits_{n\to\infty}\dfrac{\#(E_{\delta}(\om)\cap
  \{1,...,n\} )}{n}>\mu(\Om^{(\ell)})-\de$;\\[-1mm]
\item for any two independent copies $\cX,$ $\cY$ of $\cX^\var$ with $0\le\var<\var_\de$,
$$
 {\P}\big\{ X_{n}=Y_{n}|X_0,Y_0\in S\times\{\om\}\big\}\ge
 1-\var^{\ell}b_\de(\var)\quad\forall n\in E_\de(\om)\, .
$$
\end{enumerate}
\end{lemma}

\smallskip

\begin{proof}
Fix $\var_0>0$ so small that all conclusions of Theorem
\ref{thm:uniq_inv_distri} hold whenever $0\le \var < \var_0$, and pick
any $0<\de < \mu \bigl( \Om^{(\ell)} \bigr)$. By Lemma \ref{lem4n2}, for $\mu$-a.e.\ $\om
\in \Om^{(\ell)}$,
\begin{equation}\label{lim}
\lim\nolimits_{\var\to0}\dfrac{p_\varepsilon(\omega)-
  e_{J(\om)}}{\var^\ell}=0\, ,
\end{equation}
so by Egorov's theorem there exists $C_{\delta} \subset \Om^{(\ell)}$
with $\mu (C_{\delta}) > \mu \bigl( \Om^{(\ell)} \bigr) - \de$ such that the
convergence in (\ref{lim}) is uniform on $C_{\de}$. Letting $b_{\de}
(\var) = \var + \sup_{\om \in C_{\de}} |p_{\var} (\om) -
e_{J(\om)}|/\var^{\ell}$ for $0<\var<\var_0$, note that $\lim_{\var
  \to 0} b_{\de} (\var) = 0=:b_{\de} (0)$, and
$$
|p_{\var} (\om) - e_{J(\om)}| \le \var^{\ell} (b_{\de} (\var) - \var)
\quad \forall 0\le  \var < \var_0 , \om \in C_{\de}\, .
$$
Also, by Theorem \ref{thm:uniq_inv_distri}(i) there exists $N(\om)\in
\N$ such that
$$
|p_{\cX^{\var}} (n,\om) - p_{\var} (\theta^n \om)| \le  \var^{\ell + 1}
\quad \forall n \ge  N(\om) \, .
$$
Let $E_{\de} (\om) = \{n\ge N(\om) : \theta^n\om \in C_{\de}\}$. Then,
for $\mu$-a.e.\ $\om \in \Om$ the set $E_{\de}(\om)\subset \N$
satisfies (i), and for every $n\in E_{\de}(\om)$,
$$ 
 {\P}\big\{X_{n}=(s_{J(\theta^{n}\om)},\theta^{n}\om)|X_0\in
 S\times\{\om\}\big\}  = p_{\cX^{\var}} (n,\om)_{J(\theta^n\om)} =
 1-\frac{| p_{\cX^\var}(n,\om)- e_{J(\theta^n\om)}|}{2} 
 \ge  1 - \frac{\var^{\ell} b_{\de} (\var)}{2} \, ,
$$
provided that $0\le \var<\var_0$. Consequently, for any two independent
copies $\cX$, $\cY$ of $\cX^{\var}$ with $0\le \var< \var_0$,
\begin{align*}
{\P} \big\{X_{n}=Y_{n}|X_0,Y_0\in S\times\{\om\}\big\} & \ge
{\P}\big\{X_{n}=Y_n =(s_{J(\theta^{n}\om)},\theta^{n}\om)|X_0, Y_0 \in
 S\times\{\om\}\big\}  \ge \left( 1 - \frac{\var^{\ell} b_{\de}
     (\var)}{2} \right)^2 \\
& \ge  1 -
\var^{\ell} b_{\de} (\var) 
\end{align*}
for every $n\in E_{\de} (\om)$, which proves (ii) with $\var_{\de} = \var_0$.
\end{proof}

\begin{remark}
As seen in the above proof, in Lemma \ref{thm:large-Syn} one may
stipulate that $\var_{\de}$ be independent of $\de$: Simply take
$\var_{\de} = \var_0$ with $\var_0$ so small that all conclusions of Theorem
\ref{thm:uniq_inv_distri} hold whenever $0\le \var < \var_0$. The wording
of Lemma \ref{thm:large-Syn} has been chosen for consistency with
Lemma \ref{lem:small-desyn} below, as well as Theorems \ref{thm:non-unif_syn} and \ref{thm:desyn}, where
$\var_{\de}$ does depend on $\de$.
\end{remark}

\begin{lemma}\label{lem:small-desyn}
Assume that $\mu \bigl( \Om^{(\ell-1)} \setminus\Om^{(\ell)}\bigr)>0$ for some
$1\le \ell \le m$. Then, for every
$0<\de<\mu \bigl( \Om^{(\ell - 1)}\setminus \Om^{(\ell)} \bigr)$ there exist
$\var_{\de}>0$, $c_{\de}>0$, and $F_\de:\Om\to2^{\N}$ such that for
$\mu$-a.e. $\om\in\Om$,
\smallskip
\begin{enumerate}
\item
  $\liminf\nolimits_{n\to\infty}\dfrac{\#(F_{\delta}(\om)\cap
    \{1,...,n\} )}{n}>\mu
  \bigl( \Om^{(\ell-1)}\setminus
  \Om^{(\ell)}\bigr) -\de$;\\[-1mm]
\item for any two independent copies $\cX,$ $\cY$ of $\cX^\var$ with $0\le\var<\var_\de$,
$$
 {\P}\big\{ X_{n}\neq Y_{n}|X_0,Y_0\in
 S\times\{\om\}\big\}\ge \var^{\ell}c_{\de} \quad\forall n\in F_\de(\om)\,  .
$$
\end{enumerate}
\end{lemma}

\smallskip

\begin{proof}
Fix $\var_0>0$ as in the proof of Lemma \ref{thm:large-Syn}, and pick
any $0<\de < \mu \bigl( \Om^{(\ell -1 )} \setminus \Om^{(\ell)}\bigr)$. By Lemma
\ref{lem4n2}, for $\mu$-a.e.\ $\om \in \Om^{(\ell -1)}\setminus \Om^{(\ell)}$,
\begin{equation}\label{lim2}
\lim\nolimits_{\var\to0}\dfrac{p_\varepsilon(\omega)-
  e_{J(\om)}-\varepsilon^\ell q^{(\ell)}(\omega)/\ell
  !}{\var^\ell}=0\, ,
\end{equation}
so one can choose $D_{\de} \subset \Om^{(\ell -1)}\setminus \Om^{(\ell)}$ with
$\mu (D_{\de}) > \mu \bigl( \Om^{(\ell -1 )} \setminus \Om^{(\ell)}\bigr) -\de$
such that the convergence in (\ref{lim2}) is uniform on $D_{\de}$, but
also $0<c_{\de}\le 1/6$ such that $2c_{\de} \le |q^{(\ell)} (\om)/\ell!|
\le 2/c_{\de}$ for all $\om \in D_{\de}$. Let $0 < \var_{\de} \le \min
\{ \var_0 , c_{\de}^2\}$ be so small that
$$
\left|
p_{\var} (\om) - e_{J(\om)} - \frac{\var^{\ell} q^{(\ell)} (\om)}{\ell
!}
\right| < \var^{\ell} c_{\de}^2 \quad \forall 0\le  \var < \var_{\de},
\om \in D_{\de} \, .
$$
As in the proof of Lemma \ref{thm:large-Syn}, there exists $N(\om)\in
\N$ with
$$
|p_{\cX^{\var}} (n,\om) - p_{\var} (\theta^n \om)| \le \var^{\ell + 1}
\quad \forall n \ge N(\om) \, .
$$
Then $F_{\de} (\om):= \{n\ge N(\om) : \theta^n \om \in
D_{\de}\}\subset \N$ satisfies (i) for $\mu$-a.e.\ $\om \in
\Om$. Moreover, for every $n\in F_{\de} (\om)$,
\begin{align*}
 {\P}\big\{X_{n}=(s_{J(\theta^{n}\om)},\theta^{n}\om)|X_0\in
 S\times\{\om\}\big\}  & = p_{\cX^{\var}} (n,\om)_{J(\theta^n\om)} =
 1-\frac{| p_{\cX^\var}(n,\om)- e_{J(\theta^n\om)}|}{2} \\
&
 \ge  1 - \frac{\var^{\ell + 1} + c_{\de}^2 \var^{\ell} +
   2\var^{\ell}/c_{\de}}{2} \\
& \ge  1 - \var^{\ell} (c_{\de}^2 + 1/c_{\de}) \, ,
\end{align*}
provided that $0\le \var<\var_{\de}$, but also
\begin{align*}
 {\P}\big\{X_{n}\ne (s_{J(\theta^{n}\om)},\theta^{n}\om)|X_0\in
 S\times\{\om\}\big\}  & = \frac{| p_{\cX^\var}(n,\om)-
   e_{J(\theta^n\om)}|}{2} \ge  c_{\de} \var^{\ell} - \frac{\var^{\ell +
     1} + c_{\de}^2 \var^{\ell}}{2}\\
& \ge  \var^{\ell} c_{\de} (1-c_{\de}) \, ,
\end{align*}
and consequently, for any two independent copies $\cX$, $\cY$ of
$\cX^{\var}$ with $0\le \var < \var_{\de}$,
\begin{align*}
{\P}\big\{X_{n}\ne Y_n|X_0, Y_0 \in
 S\times\{\om\}\big\}  & \ge 2 \bigl( 1 - \var^{\ell} (c_{\de}^2 +
 1/c_{\de})\bigr) \var^{\ell} c_{\de} (1-c_{\de}) \\
& = \var^{\ell} \bigl( 2c_{\de} (1-c_{\de}) - 2\var^{\ell} (1-c_{\de})
(c_{\de}^3 + 1)
\bigr) \ge 2 \var^{\ell} c_{\de} (1-c_{\de})(1-2 c_{\de}) \\
& \ge \var^{\ell} c_{\de} \, ,
\end{align*}
which establishes (ii).
\end{proof}

By combining Lemmas \ref{thm:large-Syn} and \ref{lem:small-desyn}, it is now straightforward to
provide a

\begin{proof}[Proof of Theorem \ref{thm:desyn}]
Deduce from
$$
0 < 1 - \mu (\Om_{\rm deg}) = \sum\nolimits_{i=1}^m \mu
\bigl( \Om^{(i-1)}\setminus \Om^{(i)} \bigr)
$$
that $\mu \bigl( \Om^{(i-1)}\setminus \Om^{(i)} \bigr)>0$ for at least one $1\le i
\le m$, and so
\begin{eqnarray}\label{def:L}
L:= \bigl\{ 1\le i \le m : \mu \bigl( \Om^{(i-1)}\setminus
\Om^{(i)}\bigr) >0\bigr\} \ne \varnothing\, .
\end{eqnarray}
Letting $\ell = \max \, L$ for convenience, clearly $\mu \bigl(
\Om^{(\ell - 1)}\bigr)\ge \mu \bigl( \Om^{(\ell - 1)} \setminus
\Om^{(\ell)}\bigr) >0$ and $\mu \bigl( \Om^{(\ell)}\bigr) = \mu (\Om_{\rm
  deg})<1$, as well as $\sum_{i\in L} \mu \bigl( \Om^{(i-1)} \setminus
\Om^{(i)}\bigr)= 1 - \mu \bigl( \Om^{(\ell)}\bigr)$.

Assume first that $\mu \bigl( \Om^{(\ell)} \bigr) > 0$, and hence 
$0< \mu \bigl( \Om^{(\ell - 1)} \setminus
\Om^{(\ell)}\bigr) <1$. By Lemma
\ref{lem:small-desyn}, for every $i\in L$ and $0<\de < \mu \bigl(
\Om^{(i-1)} \setminus \Om^{(i)}\bigr)$ there exist $0< \var_{\de,i}\le
1$, $c_{\de,i}>0$, and $F_{\de,i}:\Om\to 2^{\N}$ such that for
$\mu$-a.e.\ $\om\in \Om$,
$$
\liminf\nolimits_{n\to \infty} \frac{\# (F_{\de,i} (\om)
  \cap\{1,\ldots, n\})}{n} > \mu \bigl( \Om^{(i-1)} \setminus
\Om^{(i)}\bigr) - \frac{\de}{m} \, ,
$$
and for any two independent copies $\cX,$ $\cY$ of $\cX^\var$ with $0\le\var<\var_{\de,i}$,
$$
 {\P}\big\{ X_{n}\neq Y_{n}|X_0,Y_0\in
 S\times\{\om\}\big\}\ge \var^i c_{\de,i} \ge \var^{\ell} c_{\de,i}
 \quad\forall n\in F_{\de,i} (\om)\,  .
$$
The sets $F_{\de,i}$ are disjoint by construction.
With this, for every $0<\de < \min_{i\in L} \mu \bigl(
\Om^{(i-1)}\setminus \Om^{(i)}\bigr)$, let $\var_{\de}' = \min_{i\in L} \var_{\de,i}>0$ and $c_{\de}=
\min_{i\in L} c_{\de,i}>0$, as well as $F_{\de}(\om) = \bigcup_{i\in
  L} F_{\de,i}(\om)$. Then,
$$
\liminf\nolimits_{n\to \infty} \frac{\# (F_{\de} (\om)
  \cap\{1,\ldots, n\})}{n} > 
\sum\nolimits_{i\in L} \left( \mu \bigl( \Om^{(i-1)} \setminus
\Om^{(i)}\bigr) - \frac{\de}{m} \right) \ge  1 - \mu \bigl(
\Om^{(\ell)}\bigr) - \de \, .
$$
By Lemma \ref{thm:large-Syn},
for every $0<\de < \mu \bigl( \Om^{(\ell)}\bigr)$ there exist
$\var_{\de,0}>0$, $ b_{\de,0}(\var )\ge 0$ with $\lim_{\var \to 0}
b_{\de,0}(\var) = 0=b_{\de,0}(0)$, and $E_{\de} :\Om \to 2^{\N}$ such that
for $\mu$-a.e.\ $\om \in \Om$,
$$
\liminf\nolimits_{n\to \infty} \frac{\# (E_{\de} (\om) \cap \{1, \ldots ,
  n\})}{n} > \mu \bigl( \Om^{(\ell)}\bigr) - \de \, ,
$$
and for any two independent copies $\cX, \cY$ of $\cX^{\var}$ with
$0\le \var < \var_{\de,0}$,
$$
{\P}\big\{ X_{n}=Y_{n}|X_0,Y_0\in S\times\{\om\}\big\}\ge
 1-\var^{\ell}b_{\de,0} (\var) \quad\forall n\in E_\de(\om)\, .
$$
Again, $E_{\de}(\om)\cap F_{\de}(\om) =\varnothing$ by construction,
provided that $\de < \min \bigl\{ \mu \bigl( \Om^{(\ell)}\bigr), \mu
\bigl( \Om^{(i-1)}\setminus \Om^{(i)}\bigr) : i \in L\bigr\}$.
Pick $\var_{\de,0}'>0$ so small that $b_{\de,0}(\var)\le c_{\de}$
whenever $0\le \var < \var_{\de,0}'$. With $\var_{\de}:= \min
\{\var_{\de}', \var_{\de,0}, \var_{\de,0}' \}>0$ and $b_{\de}:= c_{\de}$, therefore,
all assertions of the theorem are correct with $a= \mu \bigl( \Om^{(\ell)}\bigr)$. 

Next assume that $\mu\bigl( \Om^{(\ell)}\bigr)=0$ but $0<\mu
\bigl(\Om^{(\ell - 1)} \bigr) =\mu \bigl( \Om^{(\ell - 1)} \setminus
\Om^{(\ell)}\bigr) <1$. Then $\# L \ge 2$ and $\ell \ge 2$,
hence the same argument as before applies, and Theorem
\ref{thm:desyn} holds, with $\ell$ replaced by $\ell - 1\ge 1$ and $a=\mu \bigl(
\Om^{(\ell - 1)}\bigr)$.

It remains to consider the case of $\mu\bigl( \Om^{(\ell)}\bigr)=0$
but $\mu \bigl(\Om^{(\ell - 1)} \bigr) =\mu \bigl( \Om^{(\ell - 1)} \setminus
\Om^{(\ell)}\bigr) =1$, or equivalently $\# L = 1$. As in the proof of Lemma
\ref{lem:small-desyn}, for every $0<\de < 1$ one can
choose $C_{\de}\subset \Om$ with $\mu (C_{\de})>1-\de$ and
$0<c_{\de}\le 1/6$ such that $2c_{\de} \le |q^{(\ell)}(\om)/\ell !|\le 2/c_{\de}$
for all $\om \in C_{\de}$, and with $0<\var_{\de} \le \min \{ \var_0,
c_{\de}^2\}$ sufficiently small,
$$
\left| p_{\var} (\om) - e_{J(\om)} - \frac{\var^{\ell}
    q^{(\ell)}(\om)}{\ell !} \right| \le  \var^{\ell}  c_{\de}^2 
\quad \forall 0\le \var < \var_{\de}, \om \in C_{\de} \, .
$$
Also, there exists $N(\om)\in \N$ with
$$
|p_{\cX^{\var}} (n,\om) - p_{\var} (\theta^n \om)| \le  \var^{\ell+1} \quad
\forall n \ge N(\om) \, .
$$
The set $D_{\de} (\om):= \{n \ge N(\om) : \theta^n \om \in C_{\de}\}\subset \N$ satisfies
$$
\lim\nolimits_{n\to \infty} \frac{\# (D_{\de}(\om) \cap \{1, \ldots ,  n\})}{n} = \mu (C_{\de}) > 1 - \de 
$$
for $\mu$-a.e.\ $\om \in \Om$, and for every $n \in D_{\de}(\om)$,
\begin{align*}
 {\P}\big\{X_{n} =  (s_{J(\theta^{n}\om)},\theta^{n}\om) | X_0\in  S\times\{\om\}\big\} & \ge 1 - \var^{\ell} (c_{\de}^2 + 1/c_{\de}) \, , \\
 {\P}\big\{X_{n}\ne (s_{J(\theta^{n}\om)},\theta^{n}\om) | X_0\in
 S\times\{\om\}\big\} & \ge  \var^{\ell} c_{\de} (1 - c_{\de}) \, ,
\end{align*}
provided that $0\le \var<\var_{\de}$. Consequently, for any two
independent copies $\cX,\cY$ of $\cX^{\var}$ with $0\le \var< \var_{\de}$, and with $b_{\de}:= 2
(c_{\de}^2 + 1/c_{\de})>2+c_{\de}$,
\begin{align*}
 {\P}\big\{X_{n} =  Y_n |X_0,Y_0\in
 S\times\{\om\}\big\} & \ge  \bigl( 1 - \var^{\ell} (c_{\de}^2 +
 1/c_{\de})\bigr)^2 \ge  1 - \var^{\ell} b_{\de}  \, , \\
 {\P}\big\{X_{n}\ne  Y_n |X_0, Y_0 \in
 S\times\{\om\}\big\} & \ge  2 \bigl( 1 - \var^{\ell} (c_{\de}^2 +
 1/c_{\de})\bigr) \var^{\ell} c_{\de} (1 - c_{\de}) \ge  \var^{\ell} c_{\de}  \, ,
\end{align*}
whenever $n\in D_{\de}(\om)$. In this case, all assertions of the
theorem are correct with any $0<a<1$: Simply let $E_{\de} (\om)$ be a
subset of $D_{\de}(\om)$ with density $a\mu (C_{\de})> a- a\de > a -
\de$ and take $F_{\de}(\om) = D_{\de}(\om)\setminus E_{\de}(\om)$,
with density $(1-a)\mu (C_{\de}) > 1-a - (1-a)\de> 1-a - \de$.
\end{proof}

\begin{remark}
As the above proof shows, one may stipulate $b_{\de} = c_{\de}$ in
Theorem \ref{thm:desyn}, except when $\mu \bigl( \Om^{(i-1)}\setminus
\Om^{(i)}\bigr)=1$ for some (necessarily unique) $1\le i \le m$.
\end{remark}

For a simple corollary, recall the first-order
probability dissipation $d(\om)$ from Section \ref{sec4}. 

\begin{cor}
Let $\{\cX^\var\}$ be a $C^m$-Markov perturbation $(m\ge 1)$ of a
synchronized DRN $\cX^0$. If
$\mu (\{ \om \in \Om : d(\theta^{-1} \om) =0\})<1$ then the
conclusions of Theorem \ref{thm:desyn} hold with $\ell=1$
  and $a=\mu \bigl( \Om^{(1)}\bigr)$ if $\mu \bigl(
  \Om^{(1)}\bigr)>0$, or any $0<a<1$ if $\mu \bigl( \Om^{(1)}\bigr)=1$.
\end{cor}

\begin{proof}
By Lemma \ref{lem4n4}(i), $q^{(1)}\ne 0$ unless
$d(\theta^{-1}\om)=0$. Thus $\mu \bigl( \Om^{(1)}\bigr)<1$, and the
same arguments as in the proof of Theorem \ref{thm:desyn} show that
$a=\mu\bigl( \Om^{(1)}\bigr)$, provided that $\mu \bigl(\Om^{(1)}
\bigr)>0$, and otherwise $0<a<1$ is arbitrary.
\end{proof}

\section{An example}\label{sec6n}

This short final section illustrates some of the concepts and results
of this work in the context of a concrete random network. Throughout,
fix $k\ge 2$ and recall from Section \ref{sec2} that $\cM_{1,{\rm
    det}}^+$ denotes the set of all deterministic stochastic $k\times
k$-matrices; for convenience, henceforth write $\cM_{1,{\rm det}}^+$
as $\cM$. Every $A\in \cM$ corresponds to a unique map
$T_A$ of $S = \{ s_1, \ldots, s_k\}$ into itself in a natural way, via
\begin{equation}\label{eq6nm1}
    T_{A} (s_j)=s_i \quad {\text {if and only if}} \quad
    A_{i,j}=1\, .
\end{equation}
Note that $\# T_A S = \mbox{\rm rank}\, A$. For every $\ell \in \N$
consider
$$
\cM_{[\ell]} := \{A\in \cM : \mbox{\rm rank}\, A = \ell \, \} \, .
$$
Obviously, $\cM_{[\ell]} = \varnothing$ whenever $\ell > k$, and $\cM$ is
the disjoint union of $\cM_{[1]}, \ldots , \cM_{[k]}$. Moreover, $\#\cM = k^k$
whereas 
$$
\# \cM_{[\ell]}  = \left( \!\! \begin{array}{c} k \\
    \ell \end{array}\!\! \right) \sum\nolimits_{j=1}^{\ell} (-1)^{\ell -
  j} j^k 
\left( \!\! \begin{array}{c} \ell \\ j  \end{array}\!\! \right) \quad
\forall \ell \in \N \, ;
$$
in particular, $\# \cM_{[1]} = k$ and $\# \cM_{[k]} = k!$; see also
Figure \ref{fig:rw}. Let $\nu$ be the
uniform distribution on $\cM$, that is, let $\nu (\{A\})=1/k^{k}$ for
every $A\in \cM$. (The reader will notice that all subsequent
observations remain virtually unchanged as long as, more generally,
$\nu(\{A\})>0$ for every $A\in \cM$.) With these ingredients, consider
the metric dynamical system $\Theta=(\Om,\cF,\mu,\theta)$, where the
probability space is
$$
(\Om , \cF, \mu) = \bigotimes\nolimits_{z\in\Z} (\cM , 2^{\cM}, \nu)
\, ,
$$
and the map $\theta$ is the left shift
$$
\theta \bigl( (A_z)_{z\in \Z}\bigr) = (A_{z+1})_{z\in \Z} \quad
\forall (A_z)_{z\in \Z} \in \Om \, .
$$
The dynamical system $\Theta$, an example of a {\em Bernoulli shift}, is invertible
and ergodic \cite{walters82}. In probability theory parlance, $\Theta$
provides the canonical model of a (bi-infinite) sequence of i.i.d
random variables uniformly distributed on $\cM$. Note that every
$\om\in \Om$ has the form $(A_z)_{z\in\Z}$, with $A_z \in \cM$ for all
$z$, and is often written as such in what follows.

\medskip

\begin{figure}[ht] 
\includegraphics[height=6.5cm]{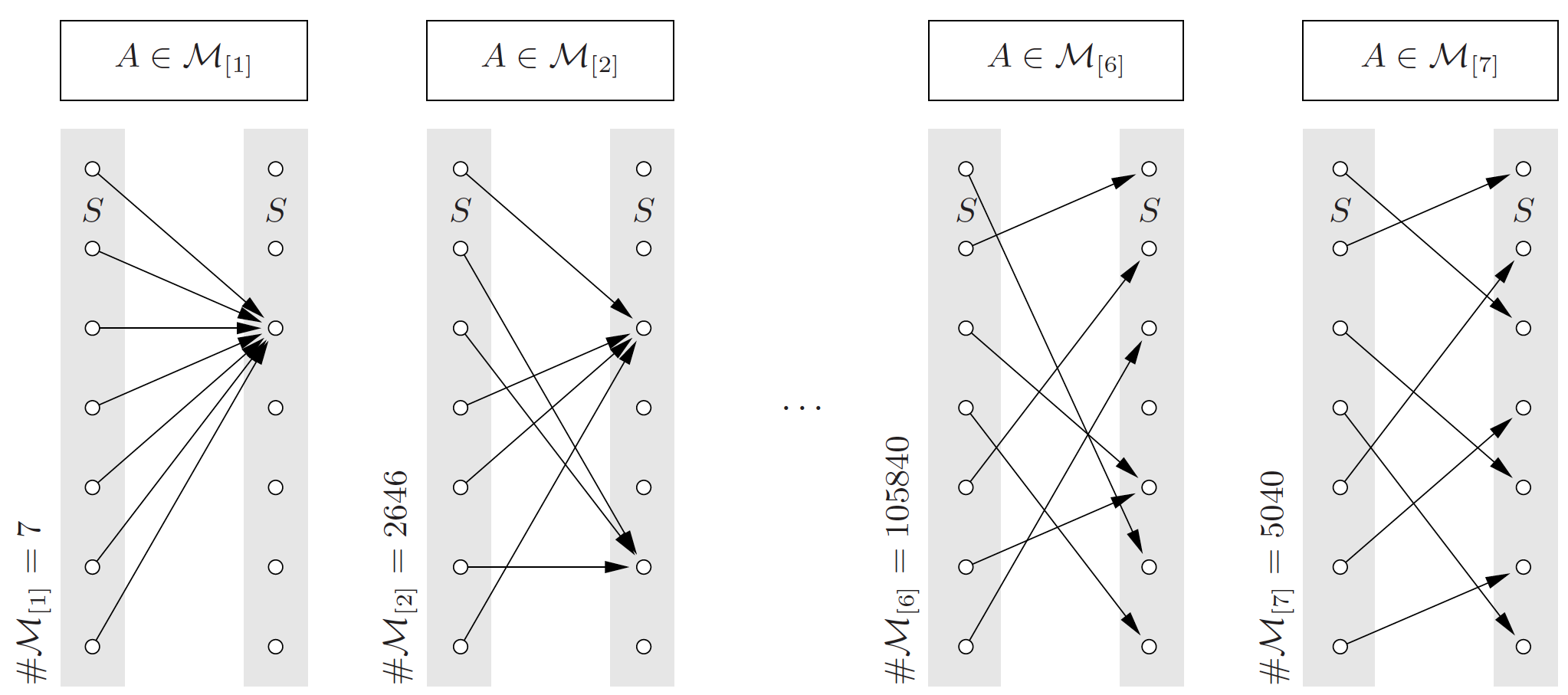}
\caption{Illustrating typical maps $T_A:S\to S$ associated with $A\in \cM$ via (\ref{eq6nm1}), for $k=7$ where $\# \cM = 823543$.}\label{fig:rw}
\end{figure}

To define a DRN $\cX^0$ over $\Theta$, let
\begin{equation}\label{eq6n0}
T_{\cX^0} (0,\om) = {\rm id}_S \, , \quad T_{\cX^0} (n,\om) =
T_{A_{n-1}} \circ \cdots \circ T_{A_0} \quad \forall n \in \N, \om  \in \Om \, ,
\end{equation}
the transition cocycle of which simply is
$$
P^0(n,\om) = \left\{
\begin{array}{ll}
I_k & \mbox{\rm if } n= 0 \, , \\
A_{n-1} \cdots A_0 & \mbox{\rm if } n\ge 1 \, .
\end{array}
\right.
$$
Given $\om \in \Om$, notice that the sequence $\bigl(
\mbox{\rm rank}\, P^0(n,\om) \bigr)_{n\in \N_0}$ in $\K$ is
non-increasing and hence convergent. In particular, if $A_m\in \cM_{[1]}$ for some $m\in \N_0$ then $\mbox{\rm
  rank}\, P^0(n,\om)=1$ for all $n>m$. Since $\nu (\cM_{[1]})>0$, this
occurs for $\mu$-a.e.\ $\om\in \Om$, and the function
$N_0^+:\Om \to \N$ given by
$$
N_0^+(\om) = \left\{
\begin{array}{ll}
\min \{n\in \N: \mbox{\rm rank}\, P^0(n,\om) = 1\} & \mbox{\rm if }
\lim_{n\to \infty} \mbox{\rm rank}\, P^0(n,\om) = 1 \, , \\[1mm]
1 & \mbox{\rm otherwise} \, ,
\end{array}
\right.
$$
is well-defined. Plainly, $\mbox{\rm rank}\, P^0(n,\om) = 1$ for $\mu$-a.e.\
$\om \in \Om$ and all $n\ge N_0^+(\om)$, and hence $N_0^+$ is a
forward synchronization time of $\cX^0$. In fact, $N_0^+\le N^+$ for
every forward synchronization time $N^+$ of $\cX^0$; in particular, $\cX^0$ is synchronized. Notice that, on the one hand,
$N_0^+$ is unbounded, since for every $n\in\N$,
$$
\mu (\{N_0^+ \ge n\}) \ge \mu \bigl( \{ \om  : A_0, \ldots,
A_{n-2} \in \cM_{[k]}\} \bigr) = \nu (\cM_{[k]})^{n-1} > 0 \, .
$$
On the other hand,
$$
\mu (\{N_0^+ \ge n\}) \le \mu \bigl( \{ \om  : A_0, \ldots,
A_{n-2} \not \in \cM_{[1]}\} \bigr) = \bigl( 1 - \nu (\cM_{[1]})\bigr)^{n-1}
\, , 
$$
and consequently $\int_{\Om} N_0^+ \, {\rm d}\mu \le
\sum_{n=1}^{\infty} \bigl( 1 - \nu (\cM_{[1]})\bigr)^{n-1} = 1/\nu
(\cM_{[1]})<\infty$.

In a completely analogous manner, $\bigl( \mbox{\rm rank}\, P^0
(n,\theta^{-n}\om)\bigr)_{n\in \N_0}$ is non-increasing for every $\om
\in \Om$, and $N_0^-:\Om\to \N$ given by
$$
N_0^-(\om) = \left\{
\begin{array}{ll}
\min \{n\in \N: \mbox{\rm rank}\, P^0(n,\theta^{-n}\om) = 1\} & \mbox{\rm if }
\lim_{n\to \infty} \mbox{\rm rank}\, P^0(n,\theta^{-n}\om) = 1 \, , \\[1mm]
1 & \mbox{\rm otherwise} \, ,
\end{array}
\right.
$$
is a pull-back synchronization time of $\cX^0$, with $N_0^- \le N^-$
for every pull-back synchronization time $N^-$ of $\cX^0$. It is
readily seen that for every $n\in \N$ the sets $\{N_0^- \ge n\}$ and
$\theta^{n-1}\{N_0^+ \ge n\}$ differ only by a $\mu$-nullset, and so
$\mu (\{N_0^- = n\}) = \mu (\{ N_0^+ = n\})$, even though clearly $\mu
(\{ N_0^- \ne N_0^+\})\ge 2\nu (\cM_{[1]}) \bigl( 1 - \nu (\cM_{[1]})\bigr)>0$. As far
as they pertain to the dynamics of $\cX^0$, the above observations yield

\begin{prop}\label{prop6_1}
The DRN $\cX^0$ defined in {\rm (\ref{eq6n0})} is synchronized, with
synchronization index $J:\Om\to \K$ given by
$e_J(\om) = \lim_{n\to \infty} A_{-1} \cdots A_{-n} e_1$
for $\mu$-a.e.\ $\om = (A_{z})_{z\in \Z} \in \Om$. An $\cF$-measurable
function $N:\Om \to \N$ is a forward (or pull-back) synchronization
time of $\cX^0$ if and only if $N\ge N_0^+$ (or $N\ge N_0^-$)
$\mu$-a.e.\ on $\Om$; in particular, $\cX^0$ is not uniformly synchronized. 
\end{prop}

To construct a Markov perturbation of $\cX^0$, recall from Section
\ref{sec2} that $\cM_0$ denotes the linear space of all real
zero-column-sum $k\times k$-matrices, and consider any function $f:\cM
\to \cM_0$. Notice that $A+\var f(A)\in \cM_1^+$ for all sufficiently
small $\var\ge 0$ if and only if
\begin{equation}\label{eq6n1}
f(A)_{i,j} (1 - 2A_{i,j}) \ge 0 \quad \forall i,j\in \K \, .
\end{equation}
Motivated by (\ref{eq6n1}), fix any $f:\cM \to \cM_0$ with the
property that for each $A\in \cM$ either $f(A)=0$, or else
\begin{equation}\label{eq6n2}
f(A)_{i,j} (1 - 2A_{i,j}) > 0 \quad \forall i,j\in \K \, ;
\end{equation}
assume w.l.o.g.\ that $|f(A)|\le 1$ as well. For every $\var \ge 0$, $n\in \N$, and
$\om \in \Om$, letting
\begin{equation}\label{eq6n3}
P_{\cX^{\var}} (n,\om) = \bigl( A_{n-1} + \min \{ \var, 1\}
f(A_{n-1})\bigr) \cdots \bigl( A_{0} + \min \{ \var, 1\}
f(A_{0})\bigr) 
\end{equation}
then yields an MRN $\cX^{\var}$, in fact a $C^{\infty}$-Markov
perturbation of $\cX^0$. Notice that by the strict inequality in
(\ref{eq6n2}), each state of $X_n$ in $S\times \{\theta^n \om\}$ has
a positive probability of leading to any state of $X_{n+1}$ in
$S\times \{\theta^{n+1}\om\}$, unless $f(A_n)=0$. Regardless of the
specific choice of $f$, Theorems \ref{thm:uniq_inv_distri} and
\ref{thm:non-unif_syn} apply. Moreover, deduce from
$P_{\cX^{\var}}^{(1)}(1,\om) = f(A_0)$ and (\ref{eq4n4n}) that
$$
P_{\cX^{\var}}^{(1)}(n,\om) =\sum\nolimits_{\ell = 0}^{n-1} A_{n-1}
\cdots A_{\ell + 1} f(A_{\ell})A_{\ell - 1} \cdots A_0 \, ,
$$
and hence the quantities $d$ and $q^{(1)}$ introduced in Section
\ref{sec4} now read
\begin{equation}\label{eq6n4}
d(\om) = f(A_0)e_{J(\om)}\, , \quad q^{(1)}(\om) =\sum\nolimits_{\ell = 0}^{N_0^-(\om)-1} A_{-1}
\cdots A_{\ell + 1- N_0^-(\om)} f(A_{\ell - N_0^-(\om)})e_{J(\theta^{\ell - N_0^-(\om)}\om)} \, .
\end{equation}
Also, the set $S_{\bullet}\subset S$ is
$S_{\bullet}(\om) = \{s_j: A_0 e_j = e_{J(\theta \om) }\}$, and hence
$S_{\bullet}(\om)\ne S$ unless $A_0 \in \cM_{[1]}$. Since either all
components of $d(\om)$ are zero or else none is, by virtue of (\ref{eq6n2}),
$$
\Om_{\bullet} = \{ \om : f(A_0) = 0 \: \mbox{\rm or}\: A_1
\in \cM_{[1]} \} \, .
$$
Finally, by utilizing (\ref{eq6n4}) the set of first-order degeneracy
$\Om^{(1)} = \{q^{(1)}=0\}$ considered in Section \ref{sec5} can be
described explicitly also.

\begin{prop}
Let $\cX^{\var}$ be the MRN defined in {\rm (\ref{eq6n3})}. Then, for
all sufficiently small $\var>0$,
$$
\Om^{(1)} = \{ \om = (A_z)_{z\in \Z} : f(A_{-1}) = \ldots =
f(A_{-N_0^-(\om)}) = 0\} \, ,
$$ 
up to a $\mu$-nullset.
\end{prop}

Notice that the value of $\mu\bigl( \Om^{(1)}\bigr)\le \nu (\{ f=0\})$
is completely determined by the set $\{f=0\}\subset \cM$. For
instance, $\mu \bigl( \Om^{(1)}\bigr)=1$ if and only if $\{f=0\}=\cM$,
and $\mu \bigl( \Om^{(1)}\bigr)= \nu (\{f=0\})$ if $\{f=0\}\subset \cM_{[1]}$.
Apart from trivial situations like these, computing the exact value
of $\mu (\Om^{(1)})$ may be a challenge. (In probability theory parlance,
$\mu\bigl( \Om^{(1)}\bigr)$ equals the probability that for a sequence
$(A_1,A_2, \ldots)$ of matrices, chosen independently and uniformly from $\cM$,
the product $A_n\cdots A_1$ attains rank $1$ {\em before\/} $f(A_n)\ne
0$ for the first time.) Regardless of the exact value, however, Theorem \ref{thm:desyn} applies with $\ell = 1$ whenever $\{f=0\}\ne
\cM$. Provided that $f:\cM\to\cM_0$ is not identically zero (and satisfies
(\ref{eq6n2}) unless $f(A)=0$), therefore, the MRN $\cX^{\var}$
defined in (\ref{eq6n3}) for
all sufficiently small $\var>0$ does exhibit the intermittency between
high-probability synchronization and low-probability desynchronization
established in that theorem.

\section*{Acknowledgement}

The authors are grateful to an anonymous referee whose thoughtful
comments helped improve the exposition.

\bibliography{myref}

\providecommand{\bysame}{\leavevmode\hbox to3em{\hrulefill}\thinspace}
\providecommand{\MR}{\relax\ifhmode\unskip\space\fi MR }
% \MRhref is called by the amsart/book/proc definition of \MR.
\providecommand{\MRhref}[2]{%
  \href{http://www.ams.org/mathscinet-getitem?mr=#1}{#2}
}
\providecommand{\href}[2]{#2}
\begin{thebibliography}{10}

\bibitem{Altman}
E.~Altman, K.~Avrachenkov, and R.~N\'u\~{n}ez Queija, \emph{Perturbation
  analysis for denumerable {M}arkov chains with application to queueing
  models}, Adv. in Appl. Probab. \textbf{36} (2004), no.~3, 839--853.

\bibitem{arnold98book}
L.~Arnold, \emph{{R}andom {D}ynamical {S}ystems}, Springer-Verlag, Berlin,
  1998.

\bibitem{arnold-ch98}
L.~Arnold and I.~Chueshov, \emph{Order-preserving random dynamical systems:
  equilibria, attractors, applications}, Dynam. Stability Systems \textbf{13}
  (1998), no.~3, 265--280.

\bibitem{arnold94positive}
L.~Arnold, V.~Gundlach, and L.~Demetrius, \emph{Evolutionary formalism for
  products of positive random matrices}, Ann. Appl. Probab. \textbf{4} (1994),
  no.~4, 859--901.

\bibitem{FGS2017}
F.~Flandoli, B.~Gess, and M.~Scheutzow, \emph{Synchronization by noise},
  Probab. Theory Related Fields \textbf{168} (2017), no.~3-4, 511--556.

\bibitem{HQWYY}
W.~Huang, H.~Qian, S.~Wang, F.~X.-F. Ye, and Y.~Yi, \emph{Synchronization in
  discrete-time, discrete-state random dynamical systems}, SIAM J. Appl. Dyn.
  Syst. \textbf{19} (2020), no.~1, 233--251.

\bibitem{huang-yi2012}
W.~Huang and Y.~Yi, \emph{On {L}yapunov exponents of continuous
  {S}chr\"{o}dinger cocycles over irrational rotations}, Proc. Amer. Math. Soc.
  \textbf{140} (2012), no.~6, 1957--1962.

\bibitem{J}
A.~Jarret, \emph{Desynchronization of random dynamical system under
  perturbation by an intrinsic noise}, arXiv:1806.07411 (2018).

\bibitem{MQY}
Y.~Ma, H.~Qian, and F.~X.-F. Ye, \emph{Stochastic dynamics: Models for
  intrinsic and extrinsic noises and their applications \rm (in {C}hinese)},
  Sci. Sin. Math. \textbf{47} (2017), 1693--1702.

\bibitem{Newman1}
J.~Newman, \emph{Necessary and sufficient conditions for stable synchronization
  in random dynamical systems}, Ergodic Theory Dynam. Systems \textbf{38}
  (2018), no.~5, 1857--1875.

\bibitem{O68}
V.~I. Oseledec, \emph{A multiplicative ergodic theorem}, Trans. Mosc. Math.
  Soc. \textbf{19} (1968), no.~2, 179--210.

\bibitem{SD}
I.~Shmulevich and E.~R. Dougherty, \emph{{P}robabilistic {B}oolean {N}etworks:
  {T}he {M}odeling and {C}ontrol of {G}ene {R}egulatory {N}etworks}, SIAM,
  2010.

\bibitem{viana2020}
M.~Viana, \emph{({D}is)continuity of {L}yapunov exponents}, Ergodic Theory
  Dynam. Systems \textbf{40} (2020), no.~3, 577--611.

\bibitem{walters82}
P.~Walters, \emph{An {I}ntroduction to {E}rgodic {T}heory}, Springer Verlag,
  1982.

\bibitem{ye-qian}
F.~X.-F. Ye and H.~Qian, \emph{Stochastic dynamics {II}: Finite random
  dynamical systems, linear representation, and entropy production}, Discrete
  Contin. Dyn. Syst. Ser. B \textbf{24} (2019), no.~8, 4341--4366.

\bibitem{ye-wang-qian}
F.~X.-F. Ye, Y.~Wang, and H.~Qian, \emph{Stochastic dynamics: Markov chains and
  random transformations}, Discrete Contin. Dyn. Syst. Ser. B \textbf{21}
  (2016), no.~7, 2337--2361.

\end{thebibliography}
\bibliographystyle{amsplain}

\end{document}